\documentclass[11pt]{article}
\usepackage{latexsym,amsfonts,amssymb,amsmath,amsthm}
\usepackage{graphicx}

\usepackage[usenames,dvipsnames]{color}
\usepackage{ulem}

\begin{document}
\setlength{\baselineskip}{16pt}

\parindent 0.5cm
\evensidemargin 0cm \oddsidemargin 0cm \topmargin 0cm \textheight 22cm \textwidth 16cm \footskip 2cm \headsep
0cm

\newtheorem{theorem}{Theorem}[section]
\newtheorem{lemma}[theorem]{Lemma}
\newtheorem{proposition}[theorem]{Proposition}
\newtheorem{definition}[theorem]{Definition}
\newtheorem{example}[theorem]{Example}
\newtheorem{corollary}[theorem]{Corollary}

\newtheorem{remark}[theorem]{Remark}

\numberwithin{equation}{section}

\def\p{\partial}
\def\I{\textit}
\def\R{\mathbb R}
\def\C{\mathbb C}
\def\u{\underline}
\def\l{\lambda}
\def\a{\alpha}
\def\O{\Omega}
\def\e{\epsilon}
\def\ls{\lambda^*}
\def\D{\displaystyle}
\def\wyx{ \frac{w(y,t)}{w(x,t)}}
\def\imp{\Rightarrow}
\def\tE{\tilde E}
\def\tX{\tilde X}
\def\tH{\tilde H}
\def\tu{\tilde u}
\def\d{\mathcal D}
\def\aa{\mathcal A}
\def\DH{\mathcal D(\tH)}
\def\bE{\bar E}
\def\bH{\bar H}
\def\M{\mathcal M}
\renewcommand{\labelenumi}{(\arabic{enumi})}

\def\disp{\displaystyle}
\def\undertex#1{$\underline{\hbox{#1}}$}
\def\card{\mathop{\hbox{card}}}
\def\sgn{\mathop{\hbox{sgn}}}
\def\exp{\mathop{\hbox{exp}}}
\def\OFP{(\Omega,{\cal F},\PP)}
\newcommand\JM{Mierczy\'nski}
\newcommand\RR{\ensuremath{\mathbb{R}}}
\newcommand\CC{\ensuremath{\mathbb{C}}}
\newcommand\QQ{\ensuremath{\mathbb{Q}}}
\newcommand\ZZ{\ensuremath{\mathbb{Z}}}
\newcommand\NN{\ensuremath{\mathbb{N}}}
\newcommand\PP{\ensuremath{\mathbb{P}}}
\newcommand\abs[1]{\ensuremath{\lvert#1\rvert}}

\newcommand\normf[1]{\ensuremath{\lVert#1\rVert_{f}}}
\newcommand\normfRb[1]{\ensuremath{\lVert#1\rVert_{f,R_b}}}
\newcommand\normfRbone[1]{\ensuremath{\lVert#1\rVert_{f, R_{b_1}}}}
\newcommand\normfRbtwo[1]{\ensuremath{\lVert#1\rVert_{f,R_{b_2}}}}
\newcommand\normtwo[1]{\ensuremath{\lVert#1\rVert_{2}}}
\newcommand\norminfty[1]{\ensuremath{\lVert#1\rVert_{\infty}}}
\newcommand{\ds}{\displaystyle}

\title{Spreading Speeds and Linear Determinacy for Two Species Competition Systems with
Nonlocal Dispersal in Periodic Habitats}

\author{Liang Kong\\
Department of Mathematical Sciences\\
University of Illinois at Springfield\\
Springfield, IL 62703\\
\\
and\\
\\
Nar Rawal and Wenxian Shen\\
Department of Mathematics and Statistics\\
Auburn University\\
Auburn University, AL 36849\\
U.S.A.}

\date{}
\maketitle

\noindent {\bf Abstract.} The current paper is concerned with the existence of spreading speeds and linear determinacy
for two species competition systems with nonlocal dispersal in time and space periodic habitats.
The notion of spreading speed intervals for such a system is first introduced via the natural features of spreading speeds.
The existence and lower bounds of spreading speed intervals are then established. When the periodic dependence of the habitat
 is only on
the time variable, the existence of a single spreading speed is proved. It also shows that, under certain conditions,
the spreading speed interval in any direction is a singleton, and, moreover, the linear determinacy holds.

\bigskip

\noindent {\bf Key words.} Competition system, nonlocal dispersal, periodic habitat, spreading speed, linear determinacy.

\bigskip

%\noindent {\bf Mathematics subject classification.}

\newpage

\section{Introduction}
\setcounter{equation}{0}

The current paper is concerned with the spatial spreading speeds of the following two species competition system with nonlocal dispersal,
\begin{equation}
\label{main-eq}
\begin{cases}
u_t=\int_{\RR^N}\kappa(y-x)u(t,y)dy-u(t,x)+u(a_1(t,x)-b_1(t,x)u-c_1(t,x)v),\quad & x\in\RR^N\cr
v_t=\int_{\RR^N}\kappa(y-x)v(t,y)dy-v(t,x)+v(a_2(t,x)-b_2(t,x)u-c_2(t,x)v),\quad  & x\in\RR^N,
\end{cases}
\end{equation}
where $u(t,x),v(t,x)$ represent the population densities of two species, $\kappa(\cdot)$ is a $C^1$
 convolution kernel supported on a ball centered at $0$ (i.e.
  $\kappa(z)>0$ if $\|z\|<r_0$ and $\kappa(z)=0$ if $\|z\|\geq r_0$ for some $r_0>0$, where
 $\|\cdot\|$ denotes the norm
 in $\RR^N$),   $\int_{\RR^N}\kappa(z)dz=1$,
 and $a_k(\cdot,\cdot)$, $b_k(\cdot,\cdot)$, $c_k(\cdot,\cdot)$ satisfy the following basic hypothesis,

 \medskip
\noindent {\bf (HB0)} {\it $a_k(t,x)$, $b_k(t,x)$, $c_k(t,x)$ $(k=1,2$) are  $C^0$ in $(t,x)\in \RR\times\RR^N$; periodic in $t$ with period $T$ and in
$x_i$ with period $p_i$, that is,
  $a_k(\cdot+T,\cdot)=a_k(\cdot,\cdot+p_i{\bf e_i})=a_k(\cdot,\cdot)$, $b_k(\cdot+T,\cdot)=b_k(\cdot,\cdot+p_i{\bf e_i})=b_k(\cdot,\cdot)$,
 $c_k(\cdot+T,\cdot)=c_k(\cdot,\cdot+p_i{\bf e_i})=c_k(\cdot,\cdot)$,
${\bf e_i}=(\delta_{i1},\delta_{i2}, \cdots,\delta_{iN})$, $\delta_{ij}=1$ if $i=j$ and $0$ if $i\not =j$,
$i,j=1,2,\cdots,N$; and $b_k(t,x)>0$, $c_k(t,x)>0$ for  $t\in\RR$ and $x\in\RR^N$.}
 \medskip

In \eqref{main-eq},  the functions $a_1$, $a_2$ represent the respective growth
rates of the two species, $b_1$, $c_2$ account for self-regulation of the respective species, and $c_1$, $b_2$
account for competition between the two species.
The periodicity of $a_k$, $b_k$, and $c_k$ reflects the time and space periodicity of the environment.

 System \eqref{main-eq} is a nonlocal dispersal counterpart of the following two species competition system with random dispersal,
\begin{equation}
\label{main-random-eq}
\begin{cases}
u_t=\Delta u+u(a_1(t,x)-b_1(t,x)u-c_1(t,x)v),\quad &x\in\RR^N\cr
v_t=\Delta v+v(a_2(t,x)-b_2(t,x)u-c_2(t,x)v),\quad &x\in\RR^N.
\end{cases}
\end{equation}

Systems \eqref{main-eq} and \eqref{main-random-eq} describe the population dynamics of two competing species
with internal interaction or dispersal.
 Classically, one assumes that the range of internal interaction of both species is infinitesimal
or the internal dispersal is random, which leads to \eqref{main-random-eq} (see \cite{CaCo},   \cite{HLM},  \cite{HMP},  \cite{Pao},
\cite{Smi}, \cite{Zha}, \cite{ZhPo}, etc.). However, in many cases  dispersal of a species is affected by long
range or nonlocal internal interaction.
We refer to \cite{BaZh}, \cite{BeCoVo1}, \cite{BeCoVo2},
\cite{ChChRo},  \cite{Cov}, \cite{Cov1}, \cite{CoDaMa}, \cite{CoDaMa1},  \cite{CoDu}, \cite{CoElRo},   \cite{Fif},
\cite{GaRo},     \cite{HMMV},
\cite{HuShVi},   \cite{KaLoSh}, \cite{LiSuWa},  \cite{ShZh1}, \cite{ShZh2}, \cite{ShZh3}, etc.,
 for the study of
evolution
models with  nonlocal  spatial
interaction. In
particular, \eqref{main-eq} arises in modeling the population dynamics of two competing species with
nonlocal spatial internal interaction (see \cite{HMMV}).

In addition to the coexistence and extinction dynamics,
spatial spreading speeds and traveling wave solutions  are among the central problems investigated for
\eqref{main-eq} and \eqref{main-random-eq}. Such problems for \eqref{main-random-eq} with space and time
independent coefficients have been widely studied.
 For example,  spreading speed/minimal wave speed and linear determinacy  for \eqref{main-random-eq} with temporally and spatially
independent coefficients
 are studied in \cite{FaZh}, \cite{GuLi}, \cite{Hos}, \cite{Hua}, \cite{HuHa},
 \cite{LeLiWe}, \cite{LiWeLe}, \cite{LiZh}, \cite{LiZh1}, \cite{Lui}, \cite{Wei1},
 \cite{Wei2}, \cite{WeLeLi}, \cite{YuZh}, etc..
  It should be pointed out that the works \cite{FaZh}, \cite{LeLiWe}, and \cite{WeLeLi}  can be applied to \eqref{main-random-eq}
  with
 spatially homogeneous and temporally periodic coefficients.  Very recently, we learned that Yu and Zhao have been studying
spatial spreading speeds and traveling wave solutions of \eqref{main-random-eq} with  coefficients that are periodic in time and space (\cite{YuZh}).

Spatial spreading speeds and traveling wave solutions for \eqref{main-eq} with  temporally and spatially independent
coefficients have also been studied in several papers. In fact,  the works  \cite{FaZh}, \cite{LeLiWe}, and \cite{WeLeLi}
can be applied to \eqref{main-eq} with  coefficients  independent of space.
However, there is little study on spatial spreading and traveling wave solutions for \eqref{main-eq} with  periodic coefficients
in both time and space.

The objective of the current paper is to study the spatial spreading speeds of \eqref{main-eq} with both
   temporally and spatially periodic coefficients. Due to the lack of compactness of the solutions, it appears to be
 difficult to adopt the construction method for spreading speeds of general
competitive or cooperative systems from \cite{FaZh}, \cite{LeLiWe}, \cite{WeLeLi},
etc. in dealing with \eqref{main-eq} with  temporally and spatially  periodic coefficients.  Therefore we will employ the natural properties of
spreading speeds to give a definition of this concept following an idea from \cite{RaShZh} and \cite{ShZh1}.
We then investigate the boundedness, lower bounds, and uniqueness  of spreading speeds.
We also study the linear determinacy for the spreading speeds.

To describe the problems studied  and the results obtained in the current paper,
let
\begin{equation}
\label{x-space}
X=\{u\in C(\RR^N,\RR)\,|\, u\,\,\,\text{is uniformly continuous and bounded}\}
\end{equation}
with the  supremum norm and
\begin{equation}
\label{x-positive-space}
X^+=\{u\in X\,|\, u(x)\ge 0\quad \forall\,\, x\in\RR^N\},\quad X^{++}=\{u\in X^+\,|\, \inf_{x\in\RR^N}u(x)>0\}.
\end{equation}
For $u,v\in X$, we define
$$
u\le v\quad (u\ll v)\quad {\rm if}\quad v-u\in X^+\quad (v-u\in X^{++}).
$$
Let
\begin{equation}
\label{x-p-space}
X_p=\{u\in C(\RR^N,\RR)\,|\, u(\cdot+p {\bf e_i})=u(\cdot)\quad  {\rm for}\,\, i=1,2,\cdots, N\}
\end{equation}
with the maximum norm
and
\begin{equation}
\label{x-p-positive-space}
X_p^+=\{u\in X_p\,|\, u(x)\ge 0\,\,\forall \,\, x\in\RR^N\},\quad X_p^{++}=\{u\in X_p^+\,|\, u(x)>0\quad \forall\,\, x\in\RR^N\}.
\end{equation}

By general semigroup theory (see \cite{Paz}), for any $u_0,v_0\in X$, \eqref{main-eq} has a unique
 (local) solution $(u(t,x;u_0,v_0)$, $v(t,x;u_0,v_0))$ with $u(0,x;u_0,v_0)=u_0(x)$
and $v(0,x;u_0,v_0)=v_0(x)$.
By comparison principle for two species competition systems
with nonlocal dispersal (see \cite[Proposition 3.1]{HeNgSh}), if $u_0,v_0\in X^+$, then $(u(t,\cdot;u_0,v_0),v(t,\cdot;u_0,v_0))$ exists  and
$(u(t,\cdot;u_0,v_0),v(t,\cdot;u_0,v_0))\in X^+\times X^+$ for all $t\ge 0$.
We remark that, if $(u_0,v_0)\in X_p\times X_p$, then $(u(t,\cdot;u_0,v_0)$, $v(t,\cdot;u_0,v_0))\in X_p\times X_p$
for $t$ in the existence interval of $(u(t,\cdot;u_0,v_0), v(t,\cdot;u_0,v_0))$.

Observe that \eqref{main-eq} contains the following two sub-systems,
\begin{equation}
\label{main-sub-eq1}
u_t=\int_{\RR^N}\kappa(y-x)u(t,y)dy-u(t,x)+u(a_1(t,x)-b_1(t,x)u),\quad x\in\RR^N,
\end{equation}
and
\begin{equation}
\label{main-sub-eq2}
v_t=\int_{\RR^N}\kappa(y-x)v(t,y)dy-v(t,x)+v(a_2(t,x)-c_2(t,x)v),\quad x\in\RR^N.
\end{equation}
The asymptotic dynamics of \eqref{main-sub-eq1} (resp. \eqref{main-sub-eq2}) is determined by the stability of its trivial
solution $u\equiv 0$ (resp. $v\equiv 0$). More precisely,
let
\begin{equation}
\label{x-pp-space} \mathcal{X}_p=\{u\in C(\RR\times \RR^N,\RR)|u(\cdot+T,\cdot)=u(\cdot,\cdot+p_i{\bf e_i})=u(\cdot,\cdot),\quad i=1,\cdots,N\}
\end{equation}
with the norm $\|u\|_{\mathcal{X}_p}=\sup_{(t,x)\in\RR\times \RR^N}|u(t,x)|$, and
\begin{equation}
\label{x-pp-positive-space} \mathcal{X}_p^+=\{u\in \mathcal{X}_p\,|\, u(t,x)\geq 0\quad \forall \,\, (t,x)\in\RR\times \RR^N\}.
\end{equation}
Let  $I$ be the identity map on $\mathcal{X}_p$, and $\mathcal{K}$, $a(\cdot,\cdot)I:\mathcal{X}_p\to \mathcal{X}_p$ be defined by
\begin{equation}
\label{k-delta-op} \big(\mathcal{K} u\big)(t,x)=\int_{\RR^N}\kappa(y-x)u(t,y)dy,
\end{equation}
\begin{equation}
\label{a-op} (a(\cdot,\cdot)Iu)(t,x)=a(t,x)u(t,x),
\end{equation}
where $\kappa(\cdot)$ is as in \eqref{main-eq} and  $a(\cdot,\cdot)\in\mathcal{X}_p$.
Let $\sigma(-\p_t +\mathcal{K}-I+a(\cdot,\cdot)I)$ be the spectrum of $-\p_t+\mathcal{K}-I+a(\cdot,\cdot)I$ acting on $\mathcal{X}_p$
and
$$\lambda_0(a):=\sup\{{\rm Re}\lambda\,|\, \lambda\in\sigma(-\p_t+\mathcal{K}-I+a(\cdot,\cdot)I)\}.
$$
We call $\lambda_0(a)$ the {\it principal spectrum point} of $-\p_t+\mathcal{K}-I+a(\cdot,\cdot)I$ acting on $\mathcal{X}_p$
(see Definition \ref{principal-spectrum-point-def}).

 Throughout the paper, we  assume that

 \medskip
\noindent{\bf (HB1)} {\it $u\equiv 0$ is a linearly unstable solution of \eqref{main-sub-eq1} in $X_p$ and
$v\equiv 0$ is a linearly unstable solution of \eqref{main-sub-eq2} in $X_p$, that is,  $\lambda_0(a_k)>0$ $(k=1,2)$.}

 \medskip

Note that $(0,0)$ is a trivial solution of \eqref{main-eq} in $X_p^+\times X_p^+$
  and that   (HB0) and (HB1) imply  that  \eqref{main-eq} has  two  semi-trivial time periodic
 solutions in $X_p^+\times X_p^+$,  $(u^*(t,\cdot),0)\in (X_p^+\setminus\{0\})\times X_p^+$ and $(0,v^*(t,\cdot))\in X_p^+\times
 (X_p^+\setminus\{0\})$
 (see Proposition \ref{semitrivial-solution-prop}).

 We also assume throughout this paper that

 \medskip

 \noindent {\bf (HB2)}
 {\it $(0,v^*)$ is linearly unstable in $X_p^+\times X_p^+$  (i.e. $\lambda_0(a_1-c_1 v^*)>0$) and $(u^*,0)$ is
 linearly and globally stable in $X_p^+\times X_p^+$
 (i.e. $\lambda_0(a_2-b_2 u^*)<0$ and for any $(u_0,v_0)\in X_p^+\times X_p^+$ with $u_0\not =0$, $u(t,x;u_0,v_0)-u^*(t,x)\to 0$ and
 $v(t,x;u_0,v_0)\to 0$ as $t\to\infty$ uniformly in $x\in\RR^N$).}

 \medskip

 The assumption (HB2) implies that the species $u$ can invade the species $v$ and the species $v$ cannot invade the species $u$.
  We remark that the following assumption (HB2$)^{'}$ on the coefficients $a_i$, $b_i$, and $c_i$ ($i=1,2$) implies (HB2)
 (see Proposition \ref{semitrivial-solution-prop} for the reasoning).

\medskip
\noindent {\bf (HB2$)^{'}$} {\it $\lambda_0(a_k)>0$ for $k=1,2$ and $a_{1L}> \frac{c_{1M}a_{2M}}{c_{2L}}$,
$a_{2M}\le \frac{a_{1L}b_{2L}}{b_{1M}}$,  where
$a_{kL}=\inf_{t\in\RR,x\in\RR^N}a_k(t,x)$, $a_{kM}=\sup_{t\in\RR,x\in\RR^N}a_k(t,x)$, and $b_{kL}$, $b_{kM}$, $c_{kL}$,
$c_{kM}$ $(k=1,2)$ are defined similarly.}

\medskip

Under the assumptions (HB0)-(HB2), spatial spreading speeds  or invading speeds from $(u^*,0)$ to $(0,v^*)$   and traveling wave solutions connecting
 $(u^*,0)$ and $(0,v^*)$ are  among most interesting dynamical problems for
\eqref{main-eq}.   The objective of this paper is to
 study the spatial spreading speeds  of \eqref{main-eq} from $(u^*,0)$ to $(0,v^*)$. In order to do so,
   we first transform \eqref{main-eq} to a cooperative system via the following standard change of variables,
\begin{equation}
\label{change-variable-eq}
\tilde u(t,x)=u(t,x),\quad \tilde v(t,x)=v^*(t,x)-v(t,x).
\end{equation}
Dropping the tilde, \eqref{main-eq} is transformed into
\begin{equation}
\label{main-eq1}
\begin{cases}
u_t=\mathcal{K}u-u+u\Big(a_1(t,x)-b_1(t,x)u-c_1(t,x)(v^*(t,x)-v)\Big)\cr
v_t=\mathcal{K}v-v+b_2(t,x)\Big(v^*(t,x)-v\Big)u+v\Big(a_2(t,x)-2c_2(t,x)v^*(t,x)+c_2(t,x)v\Big),
\end{cases}
\end{equation}
where $x\in\RR^N$, $\mathcal{K}u=\int_{\RR^N}\kappa(y-x)u(t,y)dy$ and $\mathcal{K}v=\int_{\RR^N}\kappa(y-x)v(t,y)dy$.
Observe  that the trivial solution $E_0:=(0,0)$ of \eqref{main-eq} becomes $\tilde E_0=(0,v^*)$,
the semitrivial solution $E_1:=(0,v^*)$ of \eqref{main-eq} becomes $\tilde E_1=(0,0)$, and the semitrivial
solution $E_2:=(u^*,0)$ of \eqref{main-eq} becomes $\tilde E_2=(u^*,v^*)$.
To study the spreading speeds of \eqref{main-eq} from $E_2$ to $E_1$  is then equivalent to study the spreading speeds
of \eqref{main-eq1} from $\tilde E_2$ to $\tilde E_1$. Observe also that \eqref{main-eq1} is cooperative in the region $u\ge 0$ and
$0\le v\le v^*$.

Throughout this paper, we assume (HB0)-(HB2).
We denote
$(u(t,x;u_0,v_0)$, $v(t,x;u_0,v_0))$ as the solution of \eqref{main-eq1} with $(u(0,\cdot;u_0,v_0)$, $v(0,\cdot;u_0,v_0))=(u_0,v_0)\in X\times X$.
Note that if $(u_0,v_0)\in X_p\times X_p$, then $(u(t,\cdot;u_0,v_0),v(t,\cdot;u_0,v_0))\in X_p\times X_p$ for $t$ in the
existence interval of $(u(t,\cdot;u_0,v_0),v(t,\cdot;u_0,v_0))$. Moreover, by (HB2), for any  $(u_0,v_0)\in X_p^+\times X_p^+$ with
$u_0\not = 0$ and $v_0\le v^*(0,\cdot)$,
$$
(u(t,x;u_0,v_0),v(t,x;u_0,v_0))-(u^*(t,x),v^*(t,x))\to 0
$$
as $t\to\infty$ uniformly in $x\in\RR^N$.

Let
$$
S^{N-1}=\{\xi\in\RR^N\,|\, \|\xi\|=1\}.
$$
For given $\xi\in S^{N-1}$, let
$$
X_1^+(\xi)=\{u\in X^+\,|\,  u(\cdot)\ll u^*(0,\cdot),\,\, u(x)=0\,\, {\rm for}\,\, x\cdot\xi\gg 1,\,\,
\liminf_{x\cdot\xi\to -\infty}u(x)>0\}
$$
and
$$
X_2^+(\xi)=\{v\in X^+\,|\,  v(\cdot)\ll v^*(0,\cdot),\,\,  v(x)=0\,\, {\rm for}\,\, x\cdot\xi\gg 1,\,\,
\liminf_{x\cdot\xi\to -\infty}v(x)>0\}.
$$

\begin{definition}
\label{spreading-speed-cooperative}
Let
\begin{align*}
C_{\sup}(\xi)=\Big\{c\in\RR\,|\, &\limsup_{x\cdot\xi\ge ct,t\to\infty} u^2(t,x;u_0,v_0)+v^2(t,x;u_0,v_0)=0,\\
&\,\,\,\forall\,\,
(u_0,v_0)\in X_1^+(\xi)\times X_2^+(\xi)\Big\}
\end{align*}
and
\begin{align*}
C_{\inf}(\xi)=\Big\{c\in\RR\,|\, &\limsup_{x\cdot\xi\le ct,t\to\infty}\Big[|u(t,x;u_0,v_0)-u^*(t,x)|+|v(t,x;u_0,v_0)-v^*(t,x)|\Big]=0,\\
&\,\,\forall\,\,
(u_0,v_0)\in X_1^+(\xi)\times X_2^+(\xi)\Big\}.
\end{align*}
Let
$$
c_{\sup}^*(\xi)=\begin{cases} \inf\{c\,|\, c\in C_{\sup}(\xi)\}\quad &{\rm if}\quad C_{\sup}(\xi)\not=\emptyset\cr
\infty\quad &{\rm if}\quad C_{\sup}(\xi)=\emptyset
\end{cases}
$$
and
$$
c_{\inf}^*(\xi)=\begin{cases} \sup\{c\,|\, c\in C_{\inf}(\xi)\}\quad &{\rm if}\quad C_{\inf}(\xi)\not=\emptyset\cr
-\infty\quad &{\rm if}\quad C_{\inf}(\xi)=\emptyset.
\end{cases}
$$
$[c_{\inf}^*(\xi),c_{\sup}^*(\xi)]$ is called the {\rm spreading speed interval} of \eqref{main-eq1} or \eqref{main-eq}
in the direction of $\xi$.
\end{definition}

\begin{remark}
\label{spreading-speed-rk1}
\begin{itemize}
\item[(1)] If $c\in C_{\sup}(\xi)$, then $[c,\infty)\subset C_{\sup}(\xi)$ and $C_{\inf}(\xi)\subset (-\infty,c)$.

\item[(2)] If $c\in C_{\inf}(\xi)$, then $(-\infty,c]\subset C_{\inf}(\xi)$ and $C_{\sup}(\xi)\subset (c,\infty)$.

\item[(3)] $c_{\inf}^*(\xi)\le c_{\sup}^*(\xi)$ and for any $c\in (c_{\inf}^*(\xi),c_{\sup}^*(\xi))$, there is
$(u_0,v_0)\in X_1^+(\xi)\times X_2^+(\xi)$ such that
$$
 \limsup_{x\cdot\xi\ge ct,t\to\infty} u^2(t,x;u_0,v_0)+v^2(t,x;u_0,v_0)>0
 $$
 and
 $$
\limsup_{x\cdot\xi\le ct,t\to\infty}\Big[|u(t,x;u_0,v_0)-u^*(t,x)|+|v(t,x;u_0,v_0)-v^*(t,x)|\Big]>0.
$$
\end{itemize}
\end{remark}

Let $\lambda_\xi(\mu)$ be the principal
spectrum point of the eigenvalue problem
\begin{equation}
\label{eigenvalue-eq0}
\begin{cases}
-u_t+\int_{\RR^N}e^{-\mu(y-x)\cdot\xi}\kappa(y-x)u(t,y)dy-u(t,x)+(a_1(t,x)-c_1(t,x)v^*(t,x))u(t,x)=\lambda u(t,x)\cr
u(\cdot,\cdot)\in\mathcal{X}_p
\end{cases}
\end{equation}
(see Definition \ref{principal-spectrum-point-def}  for detail).

The first two main theorems of this paper are then stated as follows.

\begin{theorem}[Finiteness and lower bound]
\label{spreading-speed-thm1} Assume (HB0)-(HB2).  For any $\xi\in S^{N-1}$,
 $[c_{\inf}^*(\xi),c_{\sup}^*(\xi)]$ is a finite interval.
Moreover,
 \begin{equation}
 \label{lower-bound-eq}
 c^*_{\inf}(\xi)\ge \inf_{\mu>0}\frac{\lambda_\xi(\mu)}{\mu}.
 \end{equation}
\end{theorem}

\begin{theorem}[Single spreading speed]
\label{spreading-speed-thm2}
Assume (HB0)-(HB2).
 If $a_k(t,x)\equiv a_k(t)$, $b_k(t,x)\equiv b_k(t)$, and
$c_k(t,x)\equiv c_k(t)$ $(k=1,2)$, then $c_{\inf}^*(\xi)=c_{\sup}^*(\xi)$ for every $\xi\in S^{N-1}$.
\end{theorem}

Observe that, by Theorem \ref{spreading-speed-thm1},  $\inf_{\mu>0}\frac{\lambda_\xi(\mu)}{\mu}$
is a lower bound of the spreading speed interval $[c_{\inf}^*(\xi),c_{\sup}^*(\xi)]$ of \eqref{main-eq1}.
It is of great interest to explore conditions such that
$c_{\inf}^*(\xi)=c_{\sup}^*(\xi)=\inf_{\mu>0}\frac{\lambda_\xi(\mu)}{\mu}$.
To this end,
we introduce the following standing assumptions.

\medskip

\noindent{\bf (HL0)} {\it $b_2(t,x)u^*(t,x)\ge c_2(t,x)v^*(t,x)$ for all $t\in\RR$ and $x\in\RR^N$.}

\medskip

\noindent {\bf (HL1)} {\it $a_1(t,x)-c_1(t,x)v^*(t,x)-a_2(t,x)+2c_2 (t,x)v^*(t,x)-b_2(t,x)v^*(t,x)\ge 0$, $b_1(t,x)\ge c_1(t,x)$, and
$b_2(t,x)\ge c_2(t,x)$ for all $t\in\RR$ and $x\in\RR^N$.}

\medskip

\noindent {\bf (HL2)} {\it $a_1(t,x)-c_1(t,x)v^*(t,x)-a_2(t,x)+2c_2 (t,x)v^*(t,x)-b_2(t,x)v^*(t,x)\frac{c_{1M}}{b_{1L}}\ge 0$
and $a_1(t,x)-c_1(t,x)v^*(t,x)-a_2(t,x)+2c_2 (t,x)v^*(t,x)-b_2(t,x)v^*(t,x)\frac{c_{2M}}{b_{2L}}\ge 0$ for all $t\in\RR$ and
$x\in\RR^N$. }

\medskip

 We remark that the following assumptions (HL0$)^{'}$, (HL1$)^{'}$, and (HL2$){'}$ on the coefficients
 $a_k$, $b_k$, $c_k$ ($k=1,2$) imply (HL0), (HL1), and (HL2), respectively.

 \medskip

 \noindent {\bf (HL0$)^{'}$} {\it $b_2(t,x)\cdot\frac{a_{1L}}{b_{1M}}\ge c_2(t,x)\cdot\frac{a_{2M}}{c_{2L}}$} for all $t\in\RR$ and
 $x\in\RR^N$.

 \medskip

 \noindent {\bf (HL1$)^{'}$} {\it $a_1(t,x)-c_1(t,x)\cdot\frac{a_{2M}}{c_{2L}}-a_2(t,x)+2c_2 (t,x)\frac{a_{2L}}{c_{2M}}-b_2(t,x)
 \cdot \frac{a_{2M}}{c_{2L}}\ge 0$, $b_1(t,x)\ge c_1(t,x)$, and
$b_2(t,x)\ge c_2(t,x)$ for all $t\in\RR$ and $x\in\RR^N$.}

\medskip

\noindent {\bf (HL2$)^{'}$} {\it $a_1(t,x)-c_1(t,x)\frac{a_{2M}}{c_{2L}}-a_2(t,x)+2c_2 (t,x)\frac{a_{2L}}{c_{2M}}
-b_2(t,x)\frac{a_{2M}}{c_{2L}}\frac{c_{1M}}{b_{1L}}\ge 0$
and $a_1(t,x)-c_1(t,x)\frac{a_{2M}}{c_{2L}}-a_2(t,x)+2c_2 (t,x)\frac{a_{2L}}{c_{2M}}-b_2(t,x)\frac{a_{2M}}{c_{2L}}\frac{c_{2M}}{b_{2L}}\ge 0$ for all $t\in\RR$ and
$x\in\RR^N$.
}
\medskip

The third main theorem of this paper is  stated as follows.

\begin{theorem}
\label{linear-determinacy-thm} Assume (HB0)-(HB2), (HL0),  and (HL1) or (HL2).
For any $\xi\in S^{N-1}$,
$c_{\inf}^*(\xi)=c_{\sup}^*(\xi)=\bar c^*_{\inf}(\xi)$.
\end{theorem}

\medskip

We point out the followings.
First, in the spatially homogeneous case, that is,
in the case $a_k(t,x)\equiv a_k(t)$, $b_k(t,x)\equiv b_k(t)$, and $c_k(t,x)\equiv c_k(t)$ ($k=1,2$),
 $c_{\inf}^*(\xi)$ and $c_{\sup}^*(\xi)$
 can be defined in the same way as   in \cite{FaZh} and \cite{WeLeLi}  (see Remark \ref{spreading-speed-rk2} in Section 3 for more detail).

 Second, when $a_k$, $b_k$ and $c_k$ ($k=1,2$) are both spatially and temporally homogeneous,
 $c_{\inf}^*(\xi)=c_{\sup}^*(\xi)$ is proved in \cite{FaZh}. In such case, the existence of traveling wave solutions
 has also been studied in \cite{FaZh} and \cite{PaLi}. The traveling wave problem in general periodic media will be studied
 somewhere else.  It remains open whether $c_{\inf}^*(\xi)=c_{\sup}^*(\xi)$ in the general periodic case.

 Third, Theorem \ref{linear-determinacy-thm} shows that the spreading speed interval is a singleton and is determined
 by the spectrum of \eqref{eigenvalue-eq0}, which is therefore referred to as {\it linear determinacy} for the spreading speed.
As mentioned in the above, the linear determinacy for the spreading speeds of \eqref{main-random-eq} with temporally and
spatially
independent coefficients has been widely studied. Our assumptions for the linear determinacy of \eqref{main-eq1} in
the case that the coefficients are independent of time and space are the same as those in the literature
for the linear determinacy of \eqref{main-random-eq} (see the following remark).

\begin{remark}
\label{linear-determinacy-rk1}
(1) When $a_k$, $b_k$, and $c_k$ ($k=1,2$) are positive constants,
$$
u^*=\frac{a_1}{b_1},\quad v^*=\frac{a_2}{c_2}.
$$
Hence the assumption (HL0) becomes
\begin{equation}
\label{HL0-eq}
\frac{a_1}{a_2}\ge \frac{b_1}{b_2},
\end{equation}
the assumption (HL1) becomes
\begin{equation}
\label{HL1-eq1}
\begin{cases}
a_1+a_2-\frac{a_2c_1}{c_2}-\frac{a_2b_2}{c_2}\ge 0\cr
b_1\ge c_1\cr
b_2\ge c_2,
\end{cases}
\end{equation}
and (HL2) becomes
\begin{equation}
\label{HL2-eq1}
\begin{cases}
a_1+a_2-\frac{a_2c_1}{c_2}-\frac{a_2b_2c_1}{b_1c_2}\ge 0\cr
a_1-\frac{a_2c_1}{c_2}\ge 0.
\end{cases}
\end{equation}

(2) In the case that
$$
a_1=r_1,\quad b_1=r_1,\quad c_1=\tilde a_1 r_1
$$
and
$$
a_2=r_2,\quad b_2=r_2\tilde a_2,\quad c_2=r_2
$$
with
$$\tilde a_1<1\le \tilde a_2,
$$
  \eqref{HL0-eq} always holds, \eqref{HL1-eq1} becomes
\begin{equation}
\label{HL1-eq2}
\frac{\tilde a_2-1}{1-\tilde a_1}\le\frac{r_1}{r_2},
\end{equation}
and \eqref{HL2-eq1} become
\begin{equation}
\label{HL2-eq2}
\frac{\tilde a_1\tilde a_2 -1}{1-\tilde a_1}\le \frac{r_1}{r_2}.
\end{equation}
By \eqref{HL2-eq2}, the assumption (HL2)
 is the same as the condition in Theorem 2.1 in \cite{LeLiWe}.
\end{remark}

  Fourth,
the techniques and theories developed for \eqref{main-eq} can be extended to two species competition systems with different nonlocal dispersal
rates in periodic habitats. To be specific and to control the length of
 the paper, we restrict the study to the case with same dispersal rates in this paper.

 Finally, the methods developed in this paper can also
be applied to two species competition systems with random dispersal or discrete dispersal in periodic habitats (see Section 5 for more detail).

 The rest of the paper is organized as follows. In Section 2, we collect some preliminary materials for the use in later sections.
 We investigate the existence of spreading speed intervals, the lower bounds of spreading speed intervals, and the existence
 of a single spreading speed in Section 3.  Theorems \ref{spreading-speed-thm1} and \ref{spreading-speed-thm2} are proved
 in this section. Section 4 is devoted to the investigation of linear determinacy of spreading speeds and to the proof of
 Theorem \ref{linear-determinacy-thm}.
 The paper is concluded with  some remarks in Section 5 on the applications of the methods developed in this paper to
  two species competition systems with random or discrete dispersals in periodic habitats.

\section{Preliminary Results}

In this section, we collect some preliminary materials for the use in later sections, including principal spectrum point and
principal eigenvalue theory for nonlocal dispersal operators with periodic coefficients; positive periodic solutions and spreading speeds
of single species models in periodic habitats; and some basic properties of two species competition systems with nonlocal dispersal.

 \subsection{Principal spectrum points and principal eigenvalues of nonlocal dispersal operators}

In this subsection, we present some principal spectrum point and principal eigenvalue theory for
time periodic nonlocal dispersal operators.

Let $\mathcal{X}_p$ be as in \eqref{x-pp-space}.
Consider the following eigenvalue problem,
\begin{equation}
\label{eigenvalue-eq}-v_t+ \big( \mathcal{K}_{\xi,\mu} -I +a(\cdot,\cdot) I
\big)v=\lambda v,\quad v\in \mathcal{X}_p,
\end{equation}
where $\xi\in S^{N-1}$, $\mu\in\RR$,  and $a(\cdot,\cdot)\in \mathcal{X}_p$.
 The operator  $a(\cdot,\cdot)I$  is
  as in \eqref{a-op}
 and $\mathcal{K}_{\xi,\mu}:\mathcal{X}_p\to \mathcal{X}_p$ is defined by
\begin{equation}
\label{k-delta-xi-mu-op}
(\mathcal{K}_{\xi,\mu}v)(t,x)=\int_{\RR^N}e^{-\mu(y-x)\cdot\xi}\kappa(y-x)v(t,y)dy.
\end{equation}
We point out that, if $u(t,x)=e^{-\mu(
x\cdot\xi-\frac{\lambda}{\mu}t)}\phi(t,x)$ with $\phi\in
\mathcal{X}_p\setminus\{0\}$ is a solution of
 the linear equation,
\begin{equation}
\label{linearization-eq0} \frac{\p u}{\p t}=\int_{\RR^N}
\kappa(y-x)u(t,y)dy-u(t,x)+a(t,x)u(t,x),\quad x\in\RR^N,
\end{equation}
 then $\lambda$ is an eigenvalue of
\eqref{eigenvalue-eq} or
$-\p_t+\mathcal{K}_{\xi,\mu}-I+a(\cdot,\cdot)I$  and $v=\phi(t,x)$ is a
corresponding eigenfunction.

 Let
$\sigma(-\p_t+ \mathcal{K}_{\xi,\mu}- I+a(\cdot,\cdot)I)$ be the spectrum of
$-\p_t+ \mathcal{K}_{\xi,\mu}- I+a(\cdot,\cdot)I$  acting on $\mathcal{X}_p$ and
\vspace{-.05in}\begin{equation*}
\lambda_0(\xi,\mu,a):=\sup\{{\rm Re}\lambda\,|\,\lambda\in \sigma( -\p_t+\mathcal{K}_{\xi,\mu}- I+a(\cdot,\cdot)I)\}.
\vspace{-.05in}\end{equation*}
Observe that if $\mu=0$, $\lambda_0(\xi,\mu,a)$ is independent of $\xi$
and hence we put
\begin{equation}
\label{lambda-delta-a}
\lambda_0(a):=\lambda_0(\xi,0,a)\quad \forall\,\, \xi\in
S^{N-1}.
\end{equation}

\begin{definition}
\label{principal-spectrum-point-def}
We call $\lambda_0(\xi,\mu,a)$  the {\rm  principal spectrum point} of $-\p_t+\mathcal{K}_{\xi,\mu}- I+a(\cdot,\cdot)I$.
$\lambda_0(\xi,\mu,a)$ is called the {\rm principal eigenvalue} of $
-\p_t+\mathcal{K}_{\xi,\mu}- I+a(\cdot,\cdot)I$  if  $\lambda_0(\xi,\mu,a)$ is an isolated
 eigenvalue of $-\p_t+\mathcal{K}_{\xi,\mu}-I+a(\cdot, \cdot)I$ with  finite algebraic multiplicity and a positive
  eigenfunction
 $v\in \mathcal{X}_p^+$,  and for every $\lambda\in
\sigma(-\p_t+ \mathcal{K}_{\xi,\mu}- I+a(\cdot,\cdot)I) \setminus
\{\lambda_0(\xi,\mu,a)\}$, ${\rm
Re}\lambda\le \lambda_0(\xi,\mu,a)$.
\end{definition}

Observe that $-\p_t + \mathcal{K}_{\xi,\mu}- I+a(\cdot,\cdot)I$ may not have a principal eigenvalue  (see an example in
\cite{ShZh1}), which reveals some essential difference between random dispersal operators and nonlocal dispersal operators.
Let
$$
\hat a(x)=\frac{1}{T}\int_0^T a(t,x)dt.
$$

\begin{proposition}
\label{PE-sufficient-prop}
 If $\hat a(\cdot)$ is $C^N$ and  the partial
derivatives of $\hat a(x)$ up to order $N-1$ at some $x_0$ are zero {\rm (we refer this to as a vanishing condition)}, where $x_0$ is such that
$\hat a(x_0)=\max_{x\in\RR^N}\hat a(x)$, then $\lambda_0(\xi,\mu,a)$ is the principal eigenvalue of $-\p_t+\mathcal{K}_{\xi,\mu}-I+a(\cdot,\cdot)I$
for all $\xi\in S^{N-1}$ and $\mu\in\RR$.
\end{proposition}

\begin{proof}
It follows from the arguments of \cite[Theorem B (1)]{RaSh}.
\end{proof}

Proposition \ref{PE-sufficient-prop} provides a useful sufficient condition for
$\lambda_0(\xi,\mu,a)$ to be the principal eigenvalue of  $-\p_t+\mathcal{K}_{\xi,\mu}-I+a(\cdot,\cdot)I$.
The following proposition shows that $\lambda_0(\xi,\mu,a)$ is
 the  principal eigenvalue of $-\p_t+\mathcal{K}_{\xi,\mu}-I+a(\cdot,\cdot)I$ for $a$ in a dense subset
 of $\mathcal{X}_p$.

\begin{proposition}
\label{PE-perturbation-prop}
 For any $\epsilon>0$ and $M>0$, there are $a^\pm(\cdot,\cdot)$ satisfying the vanishing condition in Proposition \ref{PE-sufficient-prop}
  such that
$$
a(t,x)-\epsilon\le a^-(t,x)\leq a(t,x)\leq a^+(t,x)\le a(t,x)+\epsilon
$$
for $t\in\RR$ and $x\in\RR^N$,
and
$$
|\lambda_0(\xi,\mu,a)-\lambda_0(\xi,\mu,a^\pm)|<\epsilon
$$
for $\xi\in S^{N-1}$ and $|\mu|\leq M$.
\end{proposition}

\begin{proof}
It follows from \cite[Proposition 3.7]{RaShZh}.
\end{proof}

The principal spectrum point  $\lambda_0(\xi,\mu,a)$ of \eqref{eigenvalue-eq} is closely related to the largest  growth rate of the solutions of
\begin{equation}
\label{growth-eq}
u_t=\int_{\RR^N}e^{-\mu(y-x)\cdot\xi}\kappa(y-x)u(t,y)dy-u(t,x)+a(t,x)u(t,x),\quad x\in\RR^N
\end{equation}
in $X_p$.
For given $u_0\in X_p$, let $u(t,\cdot;s,u_0)$ be the solution of \eqref{growth-eq} with $u(s,\cdot;s,u_0)=u_0(\cdot)$.
Let $\Phi(t,s;\xi,\mu,a): X_p\to X_p$ be defined by
$$
\Phi(t,s;\xi,\mu,a)u_0=u(t,\cdot;s,u_0).
$$
Note that $\Phi(t,s;\xi,\mu,a)$ is strongly monotone in the sense that for
any $t>s$ and $u_0\in X_p^{+}\setminus\{0\}$,
$\Phi(t,s;\xi,\mu,a)u_0\in X_p^{++}$ (see the arguments of \cite[Proposiion 2.2]{ShZh1}).
We have

\begin{proposition}
\label{growth-rate-prop}
For any given $\xi\in S^{N-1}$ and $\mu\in\RR$,
$$
\lambda_0(\xi,\mu,a)=\lim_{t-s\to\infty}\frac{\ln \|\Phi(t,s;\xi,\mu,a)\|}{t-s}.
$$
\end{proposition}

\begin{proof}
It follows from \cite[Propositions 3.3 and 3.10]{RaSh}.
\end{proof}

Consider the nonhomogeneous linear equation,
\begin{equation}
\label{nonhomogeneous-eq} \frac{\p u}{\p t}=\int_{\RR^N}e^{-\mu(y-x)\cdot\xi}\kappa(y-x)u(t,y)dy
-u(t,x)+a(t,x)u(t,x)+h(t,x),\quad x\in\RR^N,
\end{equation}
where $h\in\mathcal{X}_p$. We have

\begin{proposition}
\label{nonhomogeneous-prop}
If $\lambda_0(\xi,\mu,a)<0$, then for any given $h(\cdot,\cdot)\in\mathcal{X}_p$, \eqref{nonhomogeneous-eq}
has a unique entire solution $u^{**}(\cdot,\cdot)\in \mathcal{X}_p$. Moreover, $u^{**}(\cdot,\cdot)$ is  a globally stable solution
of \eqref{nonhomogeneous-eq}  with respect to perturbations in $X_p$, and if $h(t,x)\ge 0$ and $h(t,x)\not \equiv 0$, then
$u^{**}(t,\cdot)\in X^{++}_p$.
\end{proposition}

\begin{proof}
We first prove the uniqueness. Suppose that $u^{**}(t,x)$ and $v^{**}(t,x)$ are entire solutions of \eqref{nonhomogeneous-eq}
in $\mathcal{X}_p$. Let $w(t,x)=u^{**}(t,x)-v^{**}(t,x)$. Then $w(t,x)$ is an entire solution of \eqref{growth-eq} in $\mathcal{X}_p$.
We then have
$$
w(0,\cdot)=w(nT,\cdot)=\Phi(nT;\xi,\mu,a)w(0,\cdot)\quad {\rm for\,\, all}\,\, n\in\ZZ.
$$
If $w(0,\cdot)\not =0$, then by Proposition \ref{growth-rate-prop},
$$
0=\lim_{n\to\infty}\frac{\ln\|w(0,\cdot)\|}{nT}=\lim_{n\to\infty}\frac{\ln\|\Phi(nT,0;\xi,\mu,a)w(0,\cdot)\|}{nT}\le \lambda_0(\xi,\mu,a)<0,
$$
which is a contradiction. Hence $w(0,\cdot)=0$. This implies that $w(t,x)\equiv 0$ and then $u^{**}(t,x)\equiv v^{**}(t,x)$.

Next, we prove the existence. Let
\begin{equation}
\label{u-star-star-eq}
u^{**}(t,\cdot)=\int_{-\infty}^t \Phi(t,s;\xi,\mu,a)h(s,\cdot)ds.
\end{equation}
By Proposition \ref{growth-rate-prop}, for given $0<\epsilon<-\lambda_0(\xi,\mu,a)$, there is $M>0$ such that
for $t-s\ge M$,
$$
\|\Phi(t,s;\xi,\mu,a)\|<e^{(\lambda_0(\xi,\mu,a)+\epsilon)(t-s)}.
$$
Hence $u^{**}(t,\cdot)$ is  well defined  and $u^{**}(t,\cdot)\in X_p$ for all $t\in\RR$. Moreover, it is
easy to verify that $u^{**}(t,x)$ is an entire  solution of \eqref{growth-eq}. Note that
\begin{align*}
u^{**}(t+T,\cdot)&=\int_{-\infty}^{t+T}\Phi(t+T,s;\xi,\mu,a)h(s,\cdot)ds\\
&=\int_{-\infty}^t \Phi(t+T,s+T;\xi,\mu,a)h(s+T,\cdot)ds\\
&=\int_{-\infty}^t \Phi(t,s;\xi,\mu,a)h(s,\cdot)ds\\
&=u^{**}(t,\cdot).
\end{align*}
Hence $u^{**}(t,x)$ is an entire solution in $\mathcal{X}_p$.

We now prove the global stability of $u^{**}(\cdot,\cdot)$. For any given $u_0\in X_p$, let
$u(t,\cdot;u_0)$ be the solution of \eqref{nonhomogeneous-eq} with $u(0,x;u_0)=u_0(x)$.
Let $w(t,x)=u^{**}(t,x)-u(t,x;u_0)$. Then $w(t,x)$ is the solution of \eqref{growth-eq}
with $w(0,x)=u^{**}(0,x)-u_0(x)$. If $w(0,x)\not\equiv 0$, then we have
$$
\limsup_{t\to\infty}\frac{\ln\|\Phi(t,0;\xi,\mu,a)w(0,\cdot)\|}{t}\le\lambda_0(\xi,\mu,a)<0.
$$
This implies that $\|w(t,\cdot)\|\to 0$ as $t\to\infty$ exponentially. Therefore,
$u^{**}(t,x)$ is globally stable.

Finally, if $h(t,x)\ge 0$ and $h(t,x)\not\equiv 0$, by the strong monotonicity of $\Phi(t,s;\xi,\mu,a)$ and \eqref{u-star-star-eq},
$u^{**}(t,\cdot)\in X_p^{++}$ for any $t\in\RR$. The lemma is thus proved.
\end{proof}

\subsection{Positive periodic solution and spreading speeds for single species equations with nonlocal dispersal}

In this subsection, we present some results on positive periodic solution and spreading speed
for the  single species equation,
\begin{equation}
\label{main-single-eq}
u_t=\int_{\RR^N}\kappa(y-x)u(t,y)dy-u(t,x)+u(a(t,x)-b(t,x)u),\quad x\in\RR^N
\end{equation}
where $a,b\in\mathcal{X}_p$ and $b(t,x)>0$. Let $X$ be as in \eqref{x-space}.
 By general semigroup theory (see \cite{Paz}), for any $u_0\in X$, \eqref{main-single-eq} has a
unique (local) solution $u(t,x)$ with $u(0,x)=u_0(x)$. Throughout this subsection,
$u(t,x;u_0)$ denotes the solution of \eqref{main-single-eq} with
$u(0,\cdot;u_0)=u_0(\cdot)\in X$. Note that if $u_0\in X_p$, then $u(t,\cdot;u_0)\in X_p$ for
$t$ in the existence interval of $u(t,\cdot;u_0)$.

Let $\tau>0$. A continuous function $u(t,x)$ on $[0,\tau]\times\RR^N$  with
$u(t,\cdot)\in X^+$ is called a {\it super-solution} ({\it
sub-solution}) of \eqref{main-single-eq} on $[0,\tau]$ if
\begin{equation*}
u_t\ge (\le)
\int_{\RR^N}\kappa(y-x)u(t,y)dy-u(t,x)+u(a(t,x)-b(t,x)u),\quad t\in
[0,\tau],\,\,\,  x\in\RR^N.
\end{equation*}

\begin{proposition}
\label{comparison-single-species-prop}
\begin{itemize}
\item[(1)]
If $u^1(t,x)$ and $u^2(t,x)$ are bounded sub- and super-solutions of
\eqref{main-single-eq}  on $[0,\tau]$, respectively, and
$u^1(0,\cdot)\leq u^2(0,\cdot)$,  then $u^1(t,\cdot)\leq
u^2(t,\cdot)\quad {\rm for}\quad t\in [0,\tau].$

\item[(2)] For every  $u_0\in  X^+$, $ u(t,x;u_0)$ exists for all $t\geq 0$.
\end{itemize}
\end{proposition}

\begin{proof}
It follows from the arguments in \cite[Proposition 2.1]{ShZh1}.
\end{proof}

%\begin{remark}
%\label{single-species-rk1} Let
%$$
%\tilde X=\{u:\RR^N\to\RR\,|\, u \quad \text{is bounded and Lebesgue
%measurable}\}
%$$
%equipped with the norm $\|u\|_{\infty}=\sup_{x\in\RR^N}|u(x)|$ and
%$$ \tilde
%X^+=\{u\in\tilde X\,|\, u(x)\ge 0\quad {\rm for}\quad x\in\RR^N\}.
%$$
% Then for any $u_0\in \tilde X$, \eqref{main-single-eq} has also
%a unique (local) solution $u(t,x;u_0)$ with $u(0,x;u_0)=u_0(x)$. Also
%Proposition \ref{comparison-single-species-prop}(1)  holds for
%$u^1(t,\cdot),u^2(t,\cdot)\in\tilde X$ and Proposition
%\ref{comparison-single-species-prop}(2) holds for $u_0\in\tilde X^+$.
%\end{remark}

\begin{proposition}
\label{single-species-periodic-solution-prop}
Suppose that $\lambda_0(a)>0$. Then there is a unique positive periodic solution $u^*(\cdot,\cdot)\in\mathcal{X}_p^+\setminus\{0\}$
of \eqref{main-single-eq}. Moreover, for any $u_0\in {X}_p^+\setminus\{0\}$,
$$
u(t,x;u_0)-u^*(t,x)\to 0
$$
as $t\to\infty$ uniformly in $x$.
\end{proposition}

\begin{proof}
It follows from \cite[Theorem E]{RaSh}.
\end{proof}

\begin{definition}
\label{single-species-spreading-speed-interval-def}
 For a given
 vector $\xi\in S^{N-1}$, a real number $c^*_0(\xi;a,b)$ is said to be the {\rm spreading speed} of
 \eqref{main-single-eq} in the direction of $\xi$ if for any $u_0\in X^+$ satisfying that
 $$
  \liminf_{x\cdot\xi\to -\infty}u_0(x)>0\,\, \,{\rm and}\,\,\,
  u_0(x)=0\quad {\rm for \,\, all}\,\,x\in\RR^N\,\,{\rm such\,\, that}\,\, x\cdot\xi\gg 1,
 $$
 there holds
\begin{equation*}
\limsup_{t\to\infty}\sup_{x\cdot\xi \leq ct}|u(t,x;u_0)-u^*(t,x)|= 0\quad \forall\,\, c<c^*_0(\xi;a,b)
\end{equation*}
and
\begin{equation*}
  \limsup_{t\to\infty}\sup_{x\cdot\xi \geq ct}u(t,x;u_0)=0\quad \forall \,\, c>c^*_0(\xi;a,b).
\end{equation*}
\end{definition}

 \begin{proposition} [Existence of spreading speeds]
 \label{spreading-single-prop}
 Assume $\lambda_0(a)>0$. For any given $\xi\in S^{N-1}$,
 the spreading speed  $c^*_0(\xi;a,b)$ of \eqref{main-single-eq} in the direction of $\xi$ exists.
 Moreover,
 $$
 c^*_0(\xi;a,b)=\inf_{\mu>0}\frac{\lambda_0(\xi,\mu,a)}{\mu}.
 $$
 \end{proposition}

\begin{proof}
It follows from \cite[Theorem 4.1]{RaShZh}.
\end{proof}

\begin{remark}
\label{single-species-rk2} Observe that $c_0^*(\xi;a,b)$ is
independent of $b$ and we may then put
\begin{equation}
\label{single-species-speed-eq}
c_0^*(\xi,a)=c_0^*(\xi;a,b).
\end{equation}
Suppose that $a,\tilde a\in X_p$ satisfy $\lambda_0(a)>0$,
$\lambda_0(\tilde a)>0$, $\tilde a(t,x)\ge a(t,x)$ for all
$t\in\RR$, $x\in\RR^N$, and $\tilde a(t,x)\not\equiv a(t,x)$. Then
$$
c_0^*(\xi,\tilde a)>c_0^*(\xi,a).
$$
\end{remark}

\subsection{Basic properties of two species competition systems with nonlocal dispersal}

In this subsection, we present some basic properties for two species competition systems with nonlocal dispersal.

Consider \eqref{main-eq}. Let $X$ be as in \eqref{x-space}.  By general semigroup theory (see \cite{Paz}), for any $(u_0,v_0)\in X\times X$, there is a unique (local) solution $(u(t,x),v(t,x))$ of \eqref{main-eq}
with $(u(0,x),v(0,x))=(u_0(x),v_0(x))$. Throughout this subsection,  $(u(t,x;u_0,v_0),v(t,x;u_0,v_0))$ denotes
 the solution of \eqref{main-eq} with $(u(0,\cdot;u_0,v_0),v(0,\cdot;u_0,v_0))=(u_0(\cdot),v_0(\cdot))\in X\times X$,
 unless otherwise specified.

%\begin{remark}
%\label{two-species-rk1} For given $(u_0,v_0)\in \tilde X\times\tilde
%X$, where $\tilde X$ is as in Remark \ref{single-species-rk1},
%\eqref{main-eq} has also a unique (local) solution
%$(u(t,x;u_0,v_0),v(t,x;u_0,v_0))$ with
%$(u(0,x;u_0,v_0),v(0,x;u_0,v_0))=(u_0(x),v_0(x))$.
%\end{remark}

\begin{proposition}[Semitrivial solutions]
\label{semitrivial-solution-prop}
$\quad$
\begin{itemize}
\item[(1)] If $\lambda_0(a_1)>0$, then \eqref{main-eq} has a semi-trivial periodic solution $(u^*(t,x),0)$ and for
any $u_0\in X_p^+\setminus\{0\}$, $(u(t,x;u_0,0),v(t,x;u_0,0))-(u^*(t,x),0)\to 0$ as $t\to\infty$ uniformly in $x$. If,
in addition, $\lambda_0(a_2-b_2u^*)<0$, then $(u^*(t,x),0)$ is locally stable with respect to the perturbation in
$X_p\times X_p$, and if $\lambda_0(a_2-b_2u^*)>0$, then $(u^*(t,x),0)$
is unstable with respect to perturbation in $X_p\times X_p$.

\item[(2)]  If $\lambda_0(a_2)>0$, then \eqref{main-eq} has a semi-trivial periodic solution $(0,v^*(t,x))$ and for
any $v_0\in X_p^+\setminus\{0\}$, $(u(t,x;0,v_0),v(t,x;0,v_0))-(0,v^*(t,x))\to 0$ as $t\to\infty$ uniformly in $x$. If, in addition,
$\lambda_0(a_1-c_1v^*)<0$, then $(0,v^*(t,x))$ is locally stable  with respect to perturbation in $X_p\times X_p$,
 and if $\lambda_0(a_1-c_1v^*)>0$, then $(0,v^*(t,x))$ is unstable  with respect to perturbation in $X_p\times X_p$.

\item[(3)] If $\lambda_0(a_i)>0$ for $i=1,2$ and $a_{1L}> \frac{c_{1M}a_{2M}}{c_{2L}}$,
$a_{2M}\le \frac{a_{1L}b_{2L}}{b_{1M}}$,  then $(u^*(t,x),0)$ is globally stable and $(0,v^*(t,x))$ is unstable with respect to
perturbations in $X_p^+\times X_p^+$, where
$a_{kL}=\inf_{t\in\RR,x\in\RR^N}a_k(t,x)$, $a_{kM}=\sup_{t\in\RR,x\in\RR^N}a_k(t,x)$, and $b_{kL}$, $b_{kM}$, $c_{kL}$,
$c_{kM}$ $(k=1,2)$ are defined similarly.
\end{itemize}
\end{proposition}

\begin{proof}
(1) Assume $\lambda_0(a_1)>0$. Then by Proposition \ref{single-species-periodic-solution-prop}, \eqref{main-eq} has a semi-trivial periodic solution $(u^*(t,x),0)$ and for
any $u_0\in X_p^+\setminus\{0\}$, $(u(t,x;u_0,0),v(t,x;u_0,0))-(u^*(t,x),0)\to 0$ as $t\to\infty$ uniformly in $x$.
Moreover,
$$
\lambda_0(a_1-2b_1 u^*)<0.
$$

Consider the linearization of \eqref{main-eq} at $(u^*(\cdot,\cdot),0)$,
\begin{equation}
\label{main-linear-eq1}
\begin{cases}
u_t=\int_{\RR^N}\kappa(y-x)u(t,y)dy-u(t,x)+(a_1(t,x)-2b_1(t,x)u^*(t,x))u(t,x)\cr
\qquad \qquad \qquad\qquad \qquad\qquad\qquad \qquad\quad -c_1(t,x)u^*(t,x)v(t,x),\quad\qquad \qquad x\in\RR^N\cr
v_t=\int_{\RR^N}\kappa(y-x)v(t,y)dy-v(t,x)+(a_2(t,x)-b_2(t,x)u^*(t,x))v(t,x),\quad x\in\RR^N.
\end{cases}
\end{equation}
Let $P$ be the Poincar\'e map or time $T$ map of \eqref{main-linear-eq1} on $X_p\times X_p$. Observe that
$$
r(P)=e^{T\max\{\lambda_0(a_1-2b_1u^*),\lambda_0(a_2-b_2u^*)\}},
$$
where $r(P)$ is the spectral radius of $P$.
  Since $\lambda_0(a_1-2b_1 u^*)<0$, if
$\lambda_0(a_2-b_2u^*)<0$, then $r(P)<1$ and hence $(u^*,0)$ is locally stable with respect to perturbations in
$X_p\times X_p$. If $\lambda_0(a_2-b_2 u^*)>0$,
then $r(P)>1$ and hence $(u^*,0)$ is unstable with respect to perturbations in
$X_p\times X_p$.

(2) It can be proved by the same arguments as in (1).

(3) Consider
\begin{equation}
\label{simple-competition-eq}
\begin{cases}
u_t=\int_{\RR^N}\kappa(y-x)u(t,y)dy-u(t,x)+u(a_{1L}-b_{1M}u-c_{1M}v),\quad & x\in\RR^N\cr
v_t=\int_{\RR^N}\kappa(y-x)v(t,y)dy-v(t,x)+v(a_{2M}-b_{2L}u-c_{2L}v),\quad & x\in\RR^N.
\end{cases}
\end{equation}
Then $(u_-^*,0)=(\frac{a_{1L}}{b_{1M}},0)$ and $(0,v_-^*)=(0,\frac{a_{2M}}{c_{2L}})$ are
two semitrivial solutions of \eqref{simple-competition-eq}. By $\lambda_0(a_2)>0$, $a_{2M}>0$. Then
by $a_{1L}>\frac{c_{1M}a_{2M}}{c_{2L}}$ and $a_{2M}\le \frac{a_{1L}b_{2L}}{b_{1M}}$,
$(0,v_-^*)\in X_p^+\times X_p^+$ and is unstable with respect to the perturbations in
$X_p\times X_p$  and $(u_-^*,0)\in X_p^+\times X_p^+$ and is globally stable with respect to the perturbations in
$(X_p^+\setminus\{0\})\times X_p^+$.
For given $u_0\in X_p^+\setminus\{0\}$ and $v_0\in X_p^+$, let $(u_-(t,x;u_0,v_0),v_-(t,x;u_0,v_0))$ be the solution of \eqref{simple-competition-eq}
with $(u_-(0,x;u_0,v_0)$, $v_-(0,x;u_0,v_0))=(u_0(x),v_0(x))$. By comparison principle for two species competition systems
with nonlocal dispersal (see \cite[Proposition 3.1]{HeNgSh}),
$$
u(t,x;u_0,v_0)\ge u_-(t,x;u_0,v_0),\quad v(t,x;u_0,v_0)\le v_-(t,x;u_0,v_0).
$$
It then follows that
$$
\lim_{t\to\infty} v(t,x;u_0,v_0)=0
$$
uniformly in $x\in\RR^N$ and then by  Proposition \ref{single-species-periodic-solution-prop},
$$
\lim_{t\to\infty}[u(t,x;u_0,u_0)-u^*(t,x)]=0
$$
uniformly in $x\in\RR^N$. (3) thus follows.
\end{proof}

%{\rd
%\begin{remark}
%\label{cooperative-system-rk}
%Assume (HB0)-(HB2). System \eqref{main-eq1} is cooperative in the region $u\ge 0$ and
%$0\le v\le v^*$. Hence, if $(u_1(t,x),v_1(t,x))$ is a bounded sub-solution of \eqref{main-eq1} on $[0,\tau]$ with
%$u_1(t,x)\ge 0$ and $0\le v_1(t,x)\le v^*(t,x)$, and
%$(u_2(t,x),v_2(t,x))$ is a bounded super-solution of \eqref{main-eq1} on $[0,\tau]$ with $u_2(t,x)\ge 0$ and $0\le v_2(t,x)\le v^*(t,x)$, and
%if
%$u_1(0,x)\le u_2(0,x)$ and $v_1(0,x)\le v_2(0,x)$ for $x\in\RR^N$, then
%$$
%u_1(t,x)\le u_2(t,x),\quad v_1(t,x)\le v_2(t,x)
%$$
%for $t\in[0,\tau]$ and $x\in\RR^N$.
%\end{remark}
%}

\section{Spreading Speeds and the Proofs of Theorems \ref{spreading-speed-thm1} and \ref{spreading-speed-thm1}}

In this section, we investigate the spreading speeds of the cooperative system \eqref{main-eq1}
and prove Theorems \ref{spreading-speed-thm1} and \ref{spreading-speed-thm2}.

To do so,  we first prove some lemmas.
In the rest of this section, we always assume that $\xi\in S^{N-1}$ is given and fixed.
Let
$$
\begin{cases}
f(t,x,u,v)=u\Big(a_1(t,x)-b_1(t,x)u-c_1(t,x)(v^*(t,x)-v)\Big)\cr
g(t,x,u,v)=b_2(t,x)\Big(v^*(t,x)-v\Big)u+v\Big(a_2(t,x)-2c_2(t,x)v^*(t,x)+c_2(t,x)v\Big).
\end{cases}
$$
Let $(u(t,x;u_0,v_0,z),v(t,x;u_0,v_0,z))$ be the solution of
\begin{equation}
\label{main-eq2}
\begin{cases}
u_t=\mathcal{K}u-u+f(t,x+z,u,v)\cr
v_t=\mathcal{K}v-v+g(t,x+z,u,v)
\end{cases}
\end{equation}
with $(u(0,\cdot;u_0,v_0,z),v(0,\cdot;u_0,v_0,z))=(u_0(\cdot),v_0(\cdot))\in X\times X$.
Note that $(u(t,x;u_0,v_0,0)$, $v(t,x;u_0,v_0,0))=(u(t,x;u_0,v_0),v(t,x;u_0,v_0))$.

\begin{lemma}
\label{spreading-lm1}
For given $(u_1,v_1)$, $(u_2,v_2)\in X^+\times X^+$, if $0\le u_1\le u_2\le u^*(0,\cdot)$ and $0\le v_1\le v_2\le v^*(0,\cdot)$, then
$$
0\le u(t,\cdot;u_1,v_1)\le u(t,\cdot;u_2,v_2)\le u^*(t,x),\quad 0\le v(t,\cdot;u_1,v_1)\le v(t,\cdot;u_2,v_2)\le v^*(t,x).
$$
\end{lemma}

\begin{proof}
Observe that $(0,0)$ and $(u^*(t,x),v^*(t,x))$ are solutions of \eqref{main-eq1}, and $f_v(t,x,u,v)\ge 0$  for all $u\ge 0$ and
$g_u(t,x,u,v)\ge 0$ for $u\ge 0$ and $0\le v\le v^*(t,x)$.  The lemma then follows from
comparison principle for cooperative systems.
\end{proof}

\begin{lemma}
\label{spreading-lm2}
Let $(u_n,v_n)\in X^+\times X^+$ $(n=1,2,\cdots)$ and $(u_0,v_0)\in X^+\times X^+$ be given. Assume that
$u_n(\cdot)\le u^*(0,\cdot)$ and $v_n(\cdot)\le v^*(0,\cdot)$ for $n=1,2,\cdots$.
\begin{itemize}
\item[(1)]
If $(u_n,v_n)\to (u_0,v_0)$ in  compact open  topology, then
$$
(u(t,x;u_n,v_n,z),v(t,x;u_n,v_n,z))\to (u(t,x;u_0,v_0,z),v(t,x;u_0,v_0,z))
$$
as $n\to\infty$ uniformly in $z\in\RR^N$ and $(t,x)$ in bounded subsets of $\RR^+\times\RR^N$.

\item[(2)] For given $\xi\in S^{N-1}$, if $(u_n,v_n)\to (u_0,v_0)$ uniformly in bounded strips
of the form $E_K=\{x\in\RR^N\,|\, |x\cdot\xi|\le K\}$ $(K\ge 0)$, then
$$
(u(t,x;u_n,v_n,z),v(t,x;u_n,v_n,z))\to (u(t,x;u_0,v_0,z),v(t,x;u_0,v_0,z))
$$
as $n\to\infty$ uniformly in $z\in\RR^N$, $t$ in bounded subsets of $\RR^+$,  and $x$ in bounded strips
of the form $E_K=\{x\in\RR^N\,|\, |x\cdot\xi|\le K\}$ $(K\ge 0)$.
\end{itemize}
\end{lemma}
\begin{proof}
(1)
Let
$$
\begin{cases}
u_n(t,x;z)=u(t,x;u_n,v_n,z)-u(t,x;u_0,v_0,z)\cr
 v_n(t,x;z)=v(t,x;u_n,v_n,z)-v(t,x;u_0,v_0,z).
\end{cases}
$$
Then
$$
\begin{cases}
\p_t u_n(t,x;z)=\mathcal{K}u_n-u_n+a_n^1(t,x;z)u_n+b^1_n(t,x;z)v_n\cr
\p_t v_n(t,x;z)=\mathcal{K}v_n-v_n+a^2_n(t,x;z)u_n+b^2_n(t,x;z)v_n,
\end{cases}
$$
where $$a^1_n(t,x;z)=f_u(t,x+z,u_n^1(t,x;z),v_n^1(t,x;z)),
$$
$$b^1_n(t,x;z)=f_v(t,x+z,u_n^1(t,x;z),v_n^1(t,x;z)),
$$
$$a_n^2(t,x;z)=g_u(t,x+z,u_n^2(t,x;z),v_n^2(t,x;z)),
$$
and
$$b_n^2(t,x;z)=g_v(t,x+z,u_n^2(t,x;z),v_n^2(t,x;z))$$
for some $u_n^1(t,x;z)$, $u_n^2(t,x;z)$ between $u(t,x;u_n,v_n,z)$ and  $u(t,x;u_0,v_0,z)$, and some $v_n^1(t,x;z)$, $v_n^2(t,x;z)$ between $v(t,x;u_n,v_n,z)$ and
$v(t,x;u_0,v_0,z)$.

For given $\rho>0$, let
$$
Y_{\rho}=\{(u,v)\in C(\RR^n,\RR^2)\,|\, (e^{-\rho|\cdot|}u(\cdot),e^{-\rho|\cdot|}v(\cdot))\in X\times X\}
$$
equipped with the norm $\|(u,v)\|_{Y_\rho}=\sup_{x\in\RR^N} e^{-\rho|x|}(|u(x)| +|v(x)|)$.
Observe that
$$
(\mathcal{K}-I,\mathcal{K}-I):Y_\rho\to Y_\rho,
$$
$$
(\mathcal{K}-I,\mathcal{K}-I)(u,v)=(\mathcal{K}u-u,\mathcal{K}v-v),
$$
is a bounded linear operator, and $a_n^k(t,x;z)$ and $b_n^k(t,x;z)$ are bounded on $\RR^+\times\RR^N\times\RR^N$
($k=1,2$). Hence there are $M>0$ and $\omega>0$ such that
$$
\|e^{(\mathcal{K}-I,\mathcal{K}-I)t}\|_{Y_{\rho}}\le M e^{\omega t}
$$
and
$$
|a_n^k(t,x;z)|\le M,\quad |b_n^k(t,x;z)|\le M.
$$
Note that $(u_n(0,\cdot;z),v_n(0,\cdot;z))\in Y_{\rho}$ and
\begin{align*}
(u_n(t,\cdot;z),v_n(t,\cdot;z))&=e^{(\mathcal{K}-I,\mathcal{K}-I)t}(u_n(0,\cdot;z),v_n(0,\cdot;z))\\
& +\int_0^ t e^{(\mathcal{K}-I,\mathcal{K}-I)(t-s)}\Big[a_n^1(s,\cdot;z)u_n(s,\cdot;z)+b_n^1(s,\cdot;z)v_n(s,\cdot;z),\\
&\qquad \qquad\qquad
a_n^2(s,\cdot;z)u_n(s,\cdot;z)+b_n^2(s,\cdot;z)v_n(s,\cdot;z)\Big]ds.
\end{align*}
Hence
\begin{align*}
\|(u_n(t,\cdot;z),v_n(t,\cdot;z))\|_{Y_\rho}&\le M e^{\omega t}\|(u_n(0,\cdot;z),v_n(0,\cdot;z))\|_{Y_\rho}\\
&\quad
+M^2 \int_0^t e^{\omega(t-s)}\|(u_n(s,\cdot;z),v_n(s,\cdot;z))\|_{Y_\rho}ds.
\end{align*}
By Gronwall's inequality, we have
$$
\|(u_n(t,\cdot;z),v_n(t,\cdot;z))\|_{Y_\rho}\le e^{(\omega+M^2)t}M \|(u_n(0,\cdot;z),v_n(0,\cdot;z))|_{Y_\rho}.
$$
Note that
$$
\|(u_n(0,\cdot;z),v_n(0,\cdot;z))\|_{Y_\rho}\to 0
$$
as $n\to\infty$ uniformly in $z\in\RR^N$.
It then follows that
$$
(u_n(t,x;z),v_n(t,x;z))\to (0,0)
$$
as $n\to\infty$ uniformly in $z\in\RR^N$ and $(t,x)$ in bounded sets of $\RR^+\times\RR^N$.

(2) It can be proved by the similar arguments in (1) with $Y_\rho$ being replaced
by $Y_{\rho,\xi}$,
$$
Y_{\rho,\xi}=\{(u,v)\in C(\RR^N,\RR^2)\,|\, (e^{-\rho |\cdot \cdot\xi|}u(\cdot),e^{-\rho|\cdot\cdot\xi|}v(\cdot))\in X\times X\},
$$
equipped with the norm $\|(u,v)\|_{Y_{\rho,\xi}}=\sup_{x\in\RR^N} e^{-\rho|x\cdot\xi|}(|u(x)|+|v(x)|)$.
\end{proof}

\begin{lemma}
\label{spreading-lm3}
For given $u_0,v_0\in X^+$, if $\inf_{x\in\RR^N} u_0(x)>0$, $u_0(\cdot)\le u^*(0,\cdot)$, and $v_0(\cdot)\le v^*(0,\cdot)$,  then
$$
\lim_{t\to\infty} [|u(t,x;u_0,v_0)-u^*(t,x)|+|v(t,x;u_0,v_0)-v^*(t,x)|]=0
$$
uniformly in $x\in\RR^N$.
\end{lemma}

\begin{proof}
Assume that $\inf_{x\in\RR^N}u_0(x)>0$. Let $u_{0}^{\inf}=\inf_{x\in\RR^N}u_0(x)$. By Lemma \ref{spreading-lm1},
$$
u^*(t,x)\ge u(t,x;u_0,v_0)\ge u(t,x; u_{0}^{\inf},0)
$$
and
$$
v^*(t,x)\ge v(t,x;u_0,v_0)\ge v(t,x; u_{0}^{\inf},0).
$$
Note that $(u_{0}^{\inf},0)\in (X_p^+\setminus \{0\})\times X_p^+$. By (HB2), we have
$$
\lim_{t\to\infty} [|u(t,x;u_{0}^{\inf},0)-u^*(t,x)|+|v(t,x;u_{0}^{\inf},0)-v^*(t,x)|]=0
$$
uniformly in $x\in\RR^N$
and then
$$
\lim_{t\to\infty} [|u(t,x;u_0,v_0)-u^*(t,x)|+|v(t,x;u_0,v_0)-v^*(t,x)|]=0
$$
uniformly in $x\in\RR^N$.
\end{proof}

\begin{lemma}
\label{spreading-lm4}
For given $c\in\RR$ and $(u_0,v_0)\in X_1^+(\xi)\times X_2^+(\xi)$,
if $\liminf_{x\cdot\xi\le ct, t\to\infty} u(t,x;u_0,v_0)>0$, then for any $c^{'}<c$,
$$
\limsup_{x\cdot\xi\le c^{'}t,t\to\infty}[|u(t,x;u_0,v_0)-u^*(t,x)|+|v(t,x;u_0,v_0)-v^*(t,x)|]=0.
$$
\end{lemma}

\begin{proof}
First of all, by Lemma \ref{spreading-lm1},
\begin{equation}
\label{aux-eq0-1}
0\le u(t,x;u_0,v_0)\le u^*(t,x),\quad 0\le v(t,x;u_0,v_0)\le v^*(t,x)
\end{equation}
for all $t\ge 0$ and $x\in\RR^N$. For given $\delta>0$, let
$$ u_\delta(x)=u^*(0,x)-\delta,\quad  v_\delta(x)=v^*(0,x)-\delta.
$$
Observe that,
for any $\epsilon>0$, there is $0<\delta\le \epsilon$ such that
\begin{equation}
\label{aux-eq0-2}
\begin{cases}
|u(t,x; u_\delta(\cdot+z), v_\delta(\cdot+z),z)-u^*(t,x+z)|<\epsilon\cr
|u(t,x; u_\delta(\cdot+z),v_\delta(\cdot+z),z)-u^*(t,x+z)|<\epsilon
\end{cases}
\end{equation}
for all $t\in[0,T]$ and $x,z\in\RR^N$.

Let
$$\sigma=\min\{\liminf_{x\cdot\xi\le ct, t\to\infty}u(t,x;u_0,v_0),\min_{t\in\RR,x\in\RR^N}u^*(t,x),\min_{t\in\RR,x\in\RR^N}v^*(t,x)\}.
 $$
 Then there is $N_1>0$ such that
for $t\ge N_1T$ and $x\in\RR^N$ with  $x\cdot\xi \le ct$,
$$
u(t,x;u_0,v_0)\ge \sigma/2.
$$
Let
$$
\tilde u_0(x)\equiv \sigma/2,\quad \tilde v_0(x)\equiv 0.
$$
By Lemma \ref{spreading-lm3},
\begin{equation}
\label{aux-eq1}
(u(t,x;\tilde u_0,\tilde v_0),v(t,x;\tilde u_0,\tilde v_0))-(u^*(t,x),v^*(t,x))\to (0,0)
\end{equation}
as $t\to\infty$ uniformly in $x\in\RR^N$.

For any $c^{'}<c$, choose $\tilde c$ such that $c^{'}<\tilde c<c$.
Let $(\tilde u_n,\tilde v_n)\in X^+\times X^+$ be such that
\begin{equation*}
\tilde u_n(x)\begin{cases}
\le \sigma/2\quad &\forall\,\, x\in\RR^N\cr
=\sigma/2\quad &\forall\,\, x\in\RR^N,\,\,  x\cdot\xi\le (c-\tilde c)nT-1\\
=0\quad &\forall\,\, x\in\RR^N,\,\, x\cdot\xi \ge (c-\tilde c)nT,
\end{cases}
\end{equation*}
  and
  $$\tilde v_n(x)\equiv 0.
  $$
Then
\begin{equation}
\label{aux-eq2}
u(nT,x+y;u_0,v_0)\ge \tilde u_n(x),\quad v(nT,x+y;u_0,v_0)\ge \tilde v_n(x)
\end{equation}
for $y\cdot\xi\le \tilde  cnT$ and  $n\gg 1$, and
\begin{equation}
\label{aux-eq3}
(\tilde u_n(x),\tilde v_n(x))\to (\tilde u_0(x),\tilde v_0(x))
\end{equation}
as $n\to\infty$ uniformly in $x$ in bounded strips of the form $\{x\,|\, |x\cdot\xi|\le K\}$.
By  \eqref{aux-eq1}, \eqref{aux-eq3}, and Lemma \ref{spreading-lm2},  there is $N_2\ge N_1$ such that
$$
u(N_2T,x;\tilde u_n,\tilde v_n,z)\ge u^*(0,x+z)-\delta\ge \sigma/2
$$
and
$$ v(N_2T,x;\tilde u_n,\tilde v_n,z)\ge v^*(0,x+z)-\delta\ge \sigma/2
$$
for $|x\cdot\xi|\le \tilde c N_2T$,
 $z\in \RR^N$,  and $n\gg 1$. Then by \eqref{aux-eq2},  for any  $y\in \RR^N$ with $y\cdot\xi\le \tilde c nT$ and $n\gg 1$,
\begin{align*}
u((n+N_2)T,x+y;u_0,v_0)&=u(N_2T,x; u(nT,\cdot+y;u_0,v_0),v(nT,\cdot+y;u_0,v_0),y)\\
&\ge u(N_2T,x;\tilde u_n,\tilde v_n,y)\\
&\ge u^*(0,x+y)-\delta
\end{align*}
and
\begin{align*}
v((n+N_2)T,x+y;u_0,v_0)&=v(N_2T,x;  u(nT,\cdot+y;u_0,v_0),v(nT,\cdot+y;u_0,v_0),y)\\
&\ge v(N_2T,x;\tilde u_n,\tilde v_n,y)\\
&\ge v^*(0,x+y)-\delta
\end{align*}
for $|x\cdot\xi|\le \tilde c N_2T$ and $n\gg 1$.
Hence
\begin{equation}
\label{aux-eq4}
u(nT,x;u_0,v_0)\ge u^*(0,x)-\delta,\quad v(nT,x;u_0,v_0)\ge v^*(0,x)-\delta\quad \forall \,\, x\cdot\xi\le \tilde c nT,\,\, n\gg 1.
\end{equation}

Let $(\bar u_n,\bar v_n)\in X^+\times X^+$ be such that
$$
 \bar u_n(x)\begin{cases}
 \le u^*(0,x)-\delta\quad &\forall\,\, x\in\RR^N\cr
 =u^*(0,x)-\delta\quad &\forall \,\, x\in\RR^N,\,\,  x\cdot\xi\le (\tilde c-c^{'})nT-1\cr
=0\quad &\forall\,\, x\in\RR^N,\,\, x\cdot\xi\ge (\tilde c-c^{'})nT,
\end{cases}
$$
 and
 $$
 \bar v_n(x)\begin{cases}\le v^*(0,x)-\delta\quad &\forall \,\, x\in\RR^N\cr
 =v^*(0,x)-\delta\quad \,\, &\forall \,\,x\in\RR^N,\,\, x\cdot\xi\le (\tilde c-c^{'})nT-1\cr
=0\quad &\forall\,\,  x\in\RR^N,\,\, x\cdot\xi\ge (\tilde c-c^{'})nT.
\end{cases}
$$
Then
$$
u(nT,x+y;u_0,v_0)\ge \bar u_n(x),\quad v(nT,x+y;u_0,v_0)\ge \bar v_n(x)
$$
for $y\cdot\xi\le c^{'}nT$ and $n\gg 1$, and
$$
(\bar u_n(x),\bar v_n(x))\to ( u_\delta(x),v_\delta(x))
$$
as $n\to\infty$ uniformly in bounded strips of the form $|x\cdot\xi|\le K$.
By Lemma \ref{spreading-lm2} again,
\begin{equation}
\label{aux-eq6}
u(t,x;\bar u_n,\bar v_n,z)\ge u(t,x;u_\delta, v_\delta,z)-\epsilon,\quad v(t,x;\bar u_n,\bar v_n,z)\ge v(t,x; u_\delta,v_\delta,z)-\epsilon
\end{equation}
for $|x\cdot\xi|\le c^{'}T$, $z\in\RR^N$, $t\in [0,T]$, and $n\gg 1$.
Note that
\begin{align*}
u(t+nT,x+y;u_0,v_0)&=u(t,x;u(nT,\cdot+y;u_0,v_0),v(nT,\cdot+y;u_0,v_0),y)\\
&\ge u(t,x;\bar u_n,\bar v_n,y)
\end{align*}
and
\begin{align*}
v(t+nT,x+y;u_0,v_0)&=v(t,x;u(nT,\cdot+y;u_0,v_0),v(nT,\cdot+y;u_0,v_0),y)\\
&\ge v(t,x;\bar u_n,\bar v_n,y)
\end{align*}
for $t\in [0,T]$, $|x\cdot\xi|\le c^{'}T$, $y\cdot\xi\le c^{'}nT$, and $n\gg 1$. This together with
\eqref{aux-eq6} implies that
\begin{equation}
\label{aux-eq7}
\begin{cases}
u(t+nT,x+y;u_0,v_0)\ge u(t,x; u_\delta, v_\delta,y)-\epsilon\cr
 v(t+nT,x+y;u_0,v_0)\ge v(t,x;u_\delta, v_\delta,y)-\epsilon
 \end{cases}
\end{equation}
for $t\in [0,T]$, $|x\cdot \xi|\le c^{'}T$, $y\cdot\xi\le c^{'}nT$, and $n\gg 1$.
By \eqref{aux-eq0-2} and \eqref{aux-eq7}, we have that for $n\gg 1$,
\begin{equation}
\label{aux-eq8}
\begin{cases}
u(t+nT,x;u_0,v_0)\ge u^*(t+nT,x)-2\epsilon, \quad t\in[0,T],\,\, x\cdot\xi\le c^{'}(t+nT)\cr
 v(t+nT,x;u_0,v_0)\ge v^*(t+nT,x)-2\epsilon,\quad t\in [0,T],\,\, x\cdot\xi\le c^{'}(t+nT).
\end{cases}
\end{equation}
By \eqref{aux-eq0-1}, \eqref{aux-eq4}, and \eqref{aux-eq8}, for any $\epsilon>0$, there is $N>0$ such that for $t\ge NT$
and $x\cdot\xi\le c^{'}t$,
$$
-2\epsilon\le u(t,x;u_0,v_0)-u^*(t,x)\le 0,\quad -2\epsilon\le  v(t,x;u_0,v_0)-v^*(t,x)\le 0.
$$
The lemma thus follows.
\end{proof}

Let $\eta(s)$ be the function defined by
\begin{equation}
\label{eta-definition-eq} \eta(s)=\frac{1}{2}(1+\tanh\frac{s}{2}),\quad s\in\RR.
\end{equation}

Observe that
\begin{equation}
\label{eta-derivative-eq} \eta^{'}(s)=\eta(s)(1-\eta(s)),\quad s\in\RR
\end{equation}
and
\begin{equation}
\label{eta-second-derivative-eq} \eta^{''}(s)=\eta(s)(1-\eta(s))(1-2\eta(s)),\quad s\in\RR.
\end{equation}

\begin{lemma}
\label{spreading-lm5}
 There is $C_0>0$ such that for every $C\geq C_0$ and  every $\xi\in
S^{N-1}$, $(u^+(t,x;C)$, $v^+(t,x;C))$ is a super-solution of \eqref{main-eq1} on $[0,\infty)$, where
$$u^+(t,x;C)=u^*(t,x)(1-\eta(x\cdot \xi-Ct)),
$$
and
$$
v^+(t,x;C)=v^*(t,x)(1-\eta(x\cdot \xi-Ct)).
$$
\end{lemma}

\begin{proof}
Recall that
$$
f(t,x,u,v)=u\Big(a_1(t,x)-b_1(t,x)u-c_1(t,x)(v^*(t,x)-v)\Big).
$$
 By a direct calculation, we have
\begin{align*}
&f(t,x,u^*(t,x),v^*(t,x))(1-\eta(x\cdot \xi-Ct))\\
 &\quad -
f\bigl(t,x,u^*(t,x)(1-\eta(x\cdot \xi-Ct)),v^*(t,x)(1-\eta(x\cdot \xi-Ct))\bigr )\\
&=\big(1-\eta(x\cdot\xi-CT)\big) u^*(t,x)\Big\{\big[ a_1(t,x)-b_1(t,x)u^*(t,x)-c_1(t,x)(v^*(t,x)-v^*(t,x))\big]\\
&\quad -\big[a_1(t,x)-b_1(t,x)u^*(t,x)(1-\eta(x\cdot\xi-Ct))-c_1(t,x)v^*(t,x)\eta(x\cdot\xi-Ct)\big]\Big\}\\
&=\big(1-\eta(x\cdot\xi-Ct)\big)u^*(t,x)\eta(x\cdot\xi-Ct) \big(c_1(t,x)v^*(t,x)-b_1(t,x)u^*(t,x)\big)\\
&=\eta^{'}(x\cdot\xi-Ct) u^*(t,x)\big (c_1(t,x)v^*(t,x)-b_1(t,x)u^*(t,x)\big).
\end{align*}
This implies that
 \begin{align*}
&u^+_t(t,x;C)-\Big[\int_{\RR^N}\kappa(y-x)u^+(t,y;C)dy-u^+(t,x;C)+f(t,x,u^+(t,x;C),v^+(t,x;C))\Big]\\
&=\eta^{'}(x\cdot \xi-Ct)\Big\{Cu^*(t,x)
 +\int_{\RR^N}\kappa(y-x)u^*(t,y)\frac{\eta(y\cdot \xi-Ct)-
\eta(x\cdot\xi-Ct)}{\eta^{'}(x\cdot\xi-Ct)}dy\\
&\quad + u^*(t,x)\big (c_1(t,x)v^*(t,x)-b_1(t,x)u^*(t,x)\big)
\Big\}.
\end{align*}
Observe that there are $M_0$ and $M_1>0$ such that
$$
u^*(t,x),\quad  v^*(t,x)\geq M_0\quad \text{for all}\quad t\geq 0,\quad x\in\RR^N,
$$
$$
|\frac{\eta(y\cdot \xi-Ct)- \eta(x\cdot\xi-Ct)}{\eta^{'}(x\cdot\xi-Ct)}|\leq M_1\quad \text{for all}\quad t\geq
0,\quad x,y\in\RR^N,\quad \|y-x\|\leq r_0.
$$
Therefore there is $C_1>0$ such that for every $C\geq C_1$,
$$
u^+_t(t,x;C)-\big[\int_{\RR^N}\kappa(y-x)u^+(t,y;C)dy-u^+(t,x;C)+f(t,x,u^+(t,x;C),v^+(t,x;C))\big]\ge 0
$$
for $t\ge 0$ and $x\in\RR^N$.

Similarly, we can prove that there is $C_2>0$ such that for every $C\geq C_2$,
$$
v^+_t(t,x;C)-\big[\int_{\RR^N}\kappa(y-x)v^+(t,y;C)dy-v^+(t,x;C)+g(t,x,u^+(t,x;C),v^+(t,x;C))\big]\ge 0
$$
for $t\ge 0$ and $x\in\RR^N$. The lemma then follows with $C_0=\max\{C_1,C_2\}$.
\end{proof}

We now prove Theorem \ref{spreading-speed-thm1}.

\begin{proof}[Proof of Theorem \ref{spreading-speed-thm1}]
 First, consider
\begin{equation}
\label{aux-main-eq}
\bar u_t=\mathcal{K}\bar u-\bar u+\bar u(a_1(t,x)-c_1(t,x)v^*(t,x)-b_1(t,x)\bar u),\quad x\in\RR^N.
\end{equation}
For given $u_0\in X^+$, let $\bar u(t,x;u_0)$ be the solution of \eqref{aux-main-eq}
with $u(0,x;u_0)=u_0(x)$.
Let
\begin{equation*}
\bar c_{\inf}^*(\xi)=\inf_{\mu>0}\frac{\lambda_\xi(\mu)}{\mu}.
\end{equation*}
By Proposition \ref{comparison-single-species-prop} and Lemma  \ref{spreading-lm1}, for any $(u_0,v_0)\in X_1^+(\xi)\times X_2^+(\xi)$,
$$
u(t,x;u_0,v_0)\ge \bar u(t,x;u_0).
$$
By Proposition \ref{spreading-single-prop}, for any $c<\bar c_{\inf}^*(\xi)$,
$$
\liminf_{x\cdot\xi\le ct,t\to\infty} \bar u(t,x;u_0)>0
$$
and hence
$$
\liminf_{x\cdot\xi\le ct,t\to\infty}u(t,x;u_0,v_0)>0.
$$
Then by Lemma \ref{spreading-lm4}, for any $c<\bar c_{\inf}^*(\xi)$,
$$
\limsup_{x\cdot\xi\le ct,t\to\infty}\big[|u(t,x;u_0,v_0)-u^*(t,x)|+|v(t,x;u_0,v_0)-v^*(t,x)|\big]=0
$$
and hence
\begin{equation}
\label{ge-eq}
c_{\inf}^*(\xi)\ge\bar c^*_{\inf}(\xi).
\end{equation}

Next, for any given $(u_0,v_0)\in X_1^+(\xi)\times X_2^+(\xi)$,
there is $N>0$ such that
$$
u_0(x)\le u^+(NT,x;C_0),\quad v_0(x)\le v^+(NT,x;C_0)\quad \forall\,\, x\in\RR^N,
$$
where $u^+(t,x;C_0)$ and $v^+(t,x;C_0)$ are as in Lemma \ref{spreading-lm5}.
Then  by comparison principle for cooperative systems,
$$
u(t,x;u_0,v_0)\le u^+(t+NT,x;C_0),\quad v(t,x;u_0,v_0)\le v^+(t+NT,x;C_0)\quad \forall\,\, t\ge 0,\,\, x\in\RR^N.
$$
This implies that  for any $c>C_0$,
$$
\limsup_{x\cdot\xi\ge ct,t\to\infty}[u(t,x;u_0,v_0)+v(t,x;u_0,v_0)]=0
$$
and hence
\begin{equation}
\label{le-eq}
c^*_{\sup}(\xi)\le C_0.
\end{equation}
Therefore $[c^*_{\inf}(\xi),c^*_{\sup}(\xi)]$ is a bounded interval.

Finally, \eqref{lower-bound-eq} follows from \eqref{ge-eq}.
\end{proof}

Before proving Theorem \ref{spreading-speed-thm2}, we make the following remark.

\begin{remark}
\label{spreading-speed-rk2}
\begin{itemize}
\item[(1)]
 In the spatially homogeneous case, that is, $a_k(t,x)\equiv a_k(t)$, $b_k(t,x)\equiv b_k(t)$, and
$c_k(t,x)\equiv c_k(t)$ $(k=1,2)$, as in \cite{FaZh} and \cite{WeLeLi}, we can introduce
two spreading speeds, $c^*_+(\xi)$ and $\bar c_+(\xi)$, for the time $T$ map or Poincar\'e map $Q$ of \eqref{main-eq1},  as follows.
Note that in such case, $u^*(t,x)\equiv u^*(t)$, $v^*(t,x)\equiv v^*(t)$, and $Q(u^*(0),v^*(0))=(u^*(0),v^*(0))$.
Let
$\tilde \phi_0(\cdot)$ and $\tilde \psi_0(\cdot):\RR\to\RR^+$ be continuous non-increasing functions with
$$
\tilde \phi_0(s)=0,\quad \tilde \psi_0(s)=0\quad \forall s\ge 0
$$
and
$$
\tilde \phi_0(-\infty)=u^*(0),\quad \tilde \psi_0(-\infty)=v^*(0).
$$
For any $c\in\RR$, let
$$
\phi_0(c,x)=\tilde\phi_0(x\cdot\xi),\quad \psi_0(c,x)=\tilde\psi(x\cdot\xi),
$$
$$\phi_n(c,x)=\max\{\phi_0(x), u(T,x+c\xi;\phi_{n-1}(c,\cdot),\psi_{n-1}(c,\cdot))\},
$$
$$
\psi_n(c,x)=\max\{\psi_0(x),v(T,x+c\xi;\phi_{n-1}(c,\cdot),\psi_{n-1}(c,\cdot))\},
$$
and
$$
\tilde \phi_n(c,s)=\phi_n(c,s\xi),\quad \tilde \psi_n(c,s)=\psi_n(c,s\xi)
$$
for $n=1,2,\cdots$.
Thanks to the spatial homogeneity, $\tilde \phi_n(c,s)$ and $\tilde \psi_n(c,s)$ are non-increasing in
$c$ and $s$. By  comparison principle for cooperative systems, $\tilde\phi_n(c,s)$ and $\tilde\psi_n(c,s)$ are non-decreasing
in $n$.
Let
$$
\tilde \phi(c,x)=\lim_{n\to\infty}\tilde \phi_n(c,x),\quad \tilde \psi(c,x)=\lim_{n\to\infty}\tilde \psi_n(c,x).
$$
Define $c_+^*(\xi)$ and $\bar c_+(\xi)$  by
$$
c_+^*(\xi)=\sup\{c\,|\, (\tilde \phi(c,\infty),\tilde \psi(c,\infty))=(u^*(0),v^*(0))\}
$$
and
$$
\bar c_+(\xi)=\sup\{c\,|\,(\tilde \phi(c,\infty),\tilde \psi(c,\infty))\not =(0,0)\}.
$$
  We remark  that, by the similar arguments as in \cite{FaZh} and \cite{WeLeLi},  $c_+^*(\xi)$ and $\bar c_+(\xi)$ are well defined.
  We also remark that
$$
c_{\inf}^*(\xi)=c_+^*(\xi),\quad c_{\sup}^*(\xi)=\bar c_+(\xi).
$$

\item[(2)] In the case that $a_k(t,x)$, $b_k(t,x)$, and $c_k(t,x)$ $(k=1,2)$ are independent of both $t$ and $x$,
it is proved in \cite[Theorem 5.3]{FaZh} that $c_{\inf}^*(\xi)=c_{\sup}^*(\xi)$.

\item[(3)] It remains open whether \eqref{main-eq} has a single spreading speed in each direction
in both temporally and spatially periodic media.
\end{itemize}
\end{remark}

\begin{proof}[Proof of Theorem \ref{spreading-speed-thm2}]
It can be proved by applying the similar arguments as in \cite[Theorem 5.3]{FaZh}. For completeness, we give a proof
in the following.

 We first consider the case that $\xi=(1,0,\cdots,0)$.

 First of all, note that, for given $(u_0,v_0)$, if
$u_0(x_1,x_2$, $\cdots,x_N)=u_0(x_1,0,\cdots,0)$ and $v_0(x_1$, $x_2,\cdots,x_N)=v_0(x_1$, $0,\cdots,0)$,
(i.e. $u_0(x)$ and $v_0(x)$ are independent of
$x_2,\cdots, x_N$), then $u(t,x;u_0,v_0)$ and $v(t,x;u_0,v_0)$ are also independent of
$x_2,\cdots,x_N$.  We then consider the following system in one space dimension induced from
\eqref{main-eq1},
\begin{equation}
\label{main-induced-eq}
\begin{cases}
u_t(t,x_1)=\int_{\RR}\tilde \kappa(y_1-x_1)u(t,y_1)dy_1-u(t,x_1)\cr
\qquad\qquad\quad   +u\Big(a_1(t)-b_1(t)u-c_1(t)(v^*(t)- v)\Big),\quad x_1\in\RR\cr
v_t(t,x_1)=\int_{\RR}\tilde\kappa(y_1-x_1)v(t,y_1)dy_1-v(t,x_1)+b_2(t)\Big(v^*(t)-v\Big) u\cr
\qquad\qquad \quad  +v\Big(a_2(t)-2c_2(t)v^*(t)+c_2(t)v\Big),\quad\quad\quad  x_1\in\RR,
\end{cases}
\end{equation}
where
$$
\tilde \kappa(x_1)=\int_{\RR^{N-1}}\kappa(x_1,x_2,\cdots,x_N)dx_2\cdots d x_N.
$$
Observe that $u^*(t,x)\equiv u^*(t)$ and $v^*(t,x)\equiv v^*(t)$, and
$\tilde E_0=(0,v^*)$, $\tilde E_1=(0,0)$, and $\tilde E_2=(u^*,v^*)$ are time periodic solutions of
\eqref{main-induced-eq}. Hence $(0,v^*(0))$, $(0,0)$, and $(u^*(0),v^*(0))$ are the
 fixed points of the time $T$ map or the Poincar\'e map $Q$ of \eqref{main-induced-eq}.
 Observe also that $[c_{\inf}^*(\xi),c_{\sup}^*(\xi)]$ is the spreading speed
interval of \eqref{main-induced-eq}, and hence as mentioned in Remark \ref{spreading-speed-rk2}, $c_{\inf}^*(\xi)$ and
$c_{\sup}^*(\xi)$ are two spreading speeds of  the  Poincar\'e map $Q$ of \eqref{main-induced-eq}.

Next, it is not difficult to see that the Poincar\'e map $Q$ of \eqref{main-induced-eq} satisfies (A1)-(A5) in \cite{FaZh}
with $\beta=(u^*(0),v^*(0))$,
$$
\mathcal{M}=\{(u,v):\RR\to\RR^2\,|\, u,v \,\,\,\text{are non-increasing and bounded functions}\} ,
$$
and
$$
\mathcal{M}_\beta=\{(u(\cdot),v(\cdot))\in\mathcal{M}\,|\, 0\le u(x)\le u^*(0),\quad 0\le v(x)\le v^*(0)\quad \forall\,\,x\in\RR\}.
$$
Assume $c_{\inf}^*(\xi)<c_{\sup}^*(\xi)$.
 Then by \cite[Theorem 3.1(1) and (3)]{FaZh}, for
any $c_{\inf}^*(\xi)\le c< c_{\sup}^*(\xi)$, there are non-increasing functions
$\Phi(x_1)$ and $\Psi(x_1)$ such that
\begin{equation}
\label{auux-eq1}
Q^n(\Phi,\Psi)(x_1)=(\Phi(x_1-cnT),\Psi(x_1-cnT))
\end{equation}
for all but countably many $x\in\RR$,
\begin{equation}
\label{auux-eq2}
\Phi(-\infty)=u^*(0),\quad \Psi(-\infty)=v^*(0)
\end{equation}
and
\begin{equation}
\label{auux-eq3}
\Phi(\infty)=0,\quad \Psi(\infty)=v^*(0).
\end{equation}
By \eqref{auux-eq2}, \eqref{auux-eq3} and the monotonicity of $\Psi(\cdot)$, we  must have $\Psi(x_1)\equiv v^*(0)$.
Hence
\begin{equation}
\label{auux-eq4}
u(nT,x_1;\Phi(\cdot))=\Phi(x_1-cnT)
\end{equation}
for all but countably many $x\in\RR$, where
 $u(t,x_1;\Phi(\cdot))$ is the solution
of
\begin{equation}
\label{single-speed-eq1}
u_t=\int_{\RR}\tilde \kappa(y_1-x_1)u(t,y_1)dy_1-u(t,x_1)+u\big(a_1(t)-b_1(t)u).
\end{equation}
But $c_{\inf}^*(\xi)<c^*_0(\xi,a_1)$, where $c_0^*(\xi,a_1)$ is the spreading speed of \eqref{single-speed-eq1}. Choose $c, c^{'}$ such that  $c_{\inf}^*(\xi)<c<c^{'}<\min\{c^*_0(a_1,\xi),c_{\sup}^*(\xi)\}$.
 By Proposition \ref{spreading-single-prop},
 \begin{equation}
 \label{auux-eq5}
 \limsup_{x_1\le c^{'}nT,n\to\infty}|u(nT,x_1;\Phi(\cdot))-u^*(nT)|=0.
 \end{equation}
 By \eqref{auux-eq4},
 $$
 \limsup_{x_1\le c^{'}nT,n\to\infty}|u(nT,x_1;\Phi(\cdot))-u^*(nT)|=\limsup_{x_1\le c^{'}nT,n\to\infty}|\Phi(x_1-cnT)-u^*(nT)|>0,
 $$
 which contradicts to \eqref{auux-eq5}. Therefore, $c_{\inf}^*(\xi)=c_{\sup}^*(\xi)$.

 For a general $\xi\in S^{n-1}$, without loss of generality, we may assume that $\xi_1\not= 0$. Then
 make the following change of space variables,
 $$
 \tilde x_1=x\cdot\xi,\quad \tilde x_2=x_2,\cdots,\tilde x_n=x_n
 $$
 and consider the following   system induced from \eqref{main-eq1},
 $$
 \begin{cases}
 u_t(t,\tilde x_1)=\int_{\RR}\tilde\kappa(\tilde y_1-\tilde x_1)u(t,\tilde y_1)d\tilde y_1-u(t,\tilde x_1)\cr
  \qquad\qquad\quad +u\Big(a_1(t)-b_1(t)u-c_1(t)(v^*(t)- v)\Big),\quad \tilde x_1\in\RR\cr
  v_t(t,\tilde x_1)=\int_{\RR}\tilde\kappa(\tilde y_1-\tilde x_1)v(t,\tilde y_1)d\tilde y_1-v(t,\tilde x_1)+b_2(t)\Big(v^*(t)-v\Big) u\cr
\qquad\qquad \quad  +v\Big(a_2(t)-2c_2(t)v^*(t)+c_2(t)v\Big),\quad\quad\quad \tilde x_1\in\RR,
 \end{cases}
 $$
 where
 $$
 \tilde \kappa(\tilde x_1)=\xi_1\int_{\RR^{N-1}}\kappa\Big(\frac{1}{\xi_1}\big(\tilde x_1-\xi_2\cdot\tilde x_2-\cdots-\xi_N\cdot\tilde x_N,\tilde x_2,\cdots,\tilde x_N) d\tilde x_2\cdots\tilde x_N.
 $$
Now by the similar arguments as in the above, $c_{\inf}^*(\xi)=c_{\sup}^*(\xi)$.
\end{proof}

\begin{remark}
\label{spreading-speed-rk4}
The concept of spreading speed intervals for \eqref{main-eq} introduced in this section
and the results and techniques developed  in this section can be extended to two species competition systems with different nonlocal dispersal rates, that is,
the following two species competition systems,
 \begin{equation}
\label{main-different-rate-eq}
\begin{cases}
u_t=d_1[\int_{\RR^N}\kappa(y-x)u(t,y)dy-u(t,x)]+u(a_1(t,x)-b_1(t,x)u-c_1(t,x)v),\quad x\in\RR^N\cr
v_t=d_2[\int_{\RR^N}\kappa(y-x)u(t,y)dy-u(t,x)]+v(a_2(t,x)-b_2(t,x)u-c_2(t,x)v),\quad x\in\RR^N.
\end{cases}
\end{equation}

\end{remark}

\section{Linear Determinacy for Spreading Speeds and the Proof of Theorem \ref{linear-determinacy-thm}}

In this section, we explore the linear determinacy for the spreading speeds  and prove Theorem \ref{linear-determinacy-thm}.
Throughout this section, we assume (HB0)-(HB2). We also assume  (HL0) and (HL1) or (HL2).
In addition, we assume that $\xi\in S^{N-1}$ is given and fixed. We first prove some lemmas.

Consider the linearization of \eqref{main-eq1} at $(0,0)$,
\begin{equation}
\label{main-linear-eq}
\begin{cases}
u_t=\mathcal{K}u-u+(a_1(t,x)-c_1(t,x)v^*(t,x))u,\quad &x\in\RR^N\cr
v_t=\mathcal{K}v-v+b_2(t,x)v^*(t,x)u+(a_2(t,x)-2c_2(t,x)v^*(t,x))v,\quad &x\in\RR^N.
\end{cases}
\end{equation}
Consider also the following associated eigenvalue problem of \eqref{main-linear-eq},
\begin{equation}
\label{main-ev-eq}
\begin{cases}
-u_t+\mathcal{K}_{\xi,\mu}u-u+(a_1(t,x)-c_1(t,x)v^*(t,x))u=\lambda u,\quad &x\in\RR^N\cr
-v_t+\mathcal{K}_{\xi,\mu}v-v+b_2(t,x)v^*(t,x)u+(a_2(t,x)-2c_2(t,x)v^*(t,x))v=\lambda v,\quad &x\in\RR^N\cr
u,v\in\mathcal{X}_P,
\end{cases}
\end{equation}
where $\mathcal{K}_{\xi,\mu}$ is as in  \eqref{k-delta-xi-mu-op}.
Recall that $\lambda_{\xi}(\mu)$ is the principal spectrum point of
\begin{equation}
\label{linear-ev-eq1}
\begin{cases}
-u_t+\mathcal{K}_{\xi,\mu}u-u+(a_1(t,x)-c_1(t,x)v^*(t,x))u=\lambda u,\quad x\in\RR^N\cr
u\in \mathcal{X}_p.
\end{cases}
\end{equation}
Let
$$
\bar c^*_{\inf}(\xi)=\inf_{\mu>0}\frac{\lambda_{\xi}(\mu)}{\mu}.
$$

Before proving Theorem \ref{linear-determinacy-thm}, we  prove some lemmas.

Note that $\frac{\lambda_{\xi}(\mu)}{\mu}\to\infty$ as $\mu\to 0+$ or $\mu\to\infty$ (see the arguments
in \cite[Proposiiton 3.4]{RaShZh}). Hence there is
 $\mu^*(\xi)$  such that
$$
\bar c^*_{\inf}(\xi)=\frac{\lambda_{\xi}(\mu^*(\xi))}{\mu^*(\xi)}.
$$

\begin{lemma}
\label{linear-determinacy-lm1}
Assume (HL1) or (HL2).
For any $\epsilon>0$, there are $\bar a_1(t,x)\ge a_1(t,x)$, $\bar c_1(t,x)\ge c_1(t,x)$, and $\bar v^*(t,x)$ such that
\begin{itemize}
\item[(1)]
$$
\bar a_1(t,x)-\bar c_1(t,x)\bar v^*(t,x)\ge a_1(t,x)-c_1(t,x)v^*(t,x);
$$

\item[(2)]
the principal eigenvalue $\bar\lambda_{\xi}(\mu^*(\xi))$ of
\begin{equation}
\label{linear-ev-eq2}
\begin{cases}
-u_t+ \mathcal{K}_{\mu^*(\xi),\xi}u-u+(\bar a_1(t,x)-\bar c_1(t,x)\bar v^*(t,x))u=\lambda u\cr
u\in \mathcal{X}_p,
\end{cases}
\end{equation}
exists;

\item[(3)]
$$
\frac{\bar \lambda_{\xi}(\mu^*(\xi))}{\mu^*(\xi)}\le \frac{ \lambda_{\xi}(\mu^*(\xi))}{\mu^*(\xi)}+\epsilon;
$$

\item[(4)]
the principal spectrum point of
\begin{equation}
\label{linear-ev-eq3}
\begin{cases}
-v_t+\mathcal{K}_{\xi,\mu}v-v+(a_2(t,x)-2c_2(t,x)v^*(t,x))v-\bar \lambda_{\xi}(\mu^*(\xi)) v=\lambda v\cr
v\in\mathcal{X}_p
\end{cases}
\end{equation}
is negative.
\end{itemize}
\end{lemma}

\begin{proof}
(1), (2), and (3) follow from Proposition \ref{PE-perturbation-prop}.

(4) By (1) and (HL1) or (HL2),
$$
\bar a_1(t,x)-\bar c_1(t,x)\bar v^*(t,x)\ge a_1(t,x)-c_1(t,x)v^*(t,x)> a_2(t,x)-2c_2(t,x)v^*(t,x)
$$
for all  $t\in\RR$ and  $x\in\RR^N$.
Note that
$$
\lambda(\mu,\xi,\bar a_1-\bar c_1 \bar v^*-\bar \lambda_{\xi}(\mu^*(\xi)))=0.
$$
Hence
$$
\lambda(\mu,\xi,a_2(t,x)-2c_2(t,x)v^*(t,x)-\bar \lambda_{\xi}(\mu^*(\xi)))<0.
$$
\end{proof}

In the following, we put $\mu=\mu^*(\xi)$.
Let  $u_{\xi,\mu}(\cdot,\cdot)\in\mathcal{X}_p$ be a positive principal eigenfunction of \eqref{linear-ev-eq2}.
By Lemma \ref{linear-determinacy-lm1} (4) and Proposition \ref{nonhomogeneous-prop},
there is a unique  time and space periodic positive solution in $\mathcal{X}_p$, denoted by $v_{\xi,\mu}$,  of
\begin{equation}
\label{linear-ev-eq4}
v_t=\mathcal{K}_{\xi,\mu}v-v+(a_2(t,x)-2c_2(t,x)v^*(t,x))v-\bar \lambda_{\xi}(\mu) v+b_2(t,x)v^*(t,x)u_{\xi,\mu}(t,x),\quad x\in\RR^N
\end{equation}
Moreover, $v_{\xi,\mu}$ is a globally  asymptotically stable solution of \eqref{linear-ev-eq4} in $X_p$.

\begin{lemma}
\label{linear-determinacy-lm2}
Assume (HL1) or (HL2). Then
 $$c_1(t,x) v_{\xi,\mu}(t,x)\le b_1(t,x) u_{\xi,\mu}(t,x)$$ and
 $$c_2(t,x)v_{\xi,\mu}(t,x)\le b_2(t,x) u_{\xi,\mu}(t,x).$$
\end{lemma}

\begin{proof}
We first assume (HL1). We claim that $v=u_{\xi,\mu}(t,x)$ is a super-solution of
\eqref{linear-ev-eq4}. In fact,
\begin{align*}
&\p_t u_{\xi,\mu}(t,x)-\Big[\mathcal{K}_{\xi,\mu}u_{\xi,\mu}-u_{\xi,\mu}+(a_2(t,x)-2c_2(t,x)v^*(t,x))u_{\xi,\mu}\\
&\quad\qquad\quad\quad\quad -\bar \lambda_{\xi}(\mu) u_{\xi,\mu}+b_2(t,x)v^*(t,x)u_{\xi,\mu}(t,x)\Big]\\
&=\mathcal{K}_{\xi,\mu}u_{\xi,\mu}-u_{\xi,\mu}+(\bar a_1-\bar c_1\bar v^*(t,x))u_{\xi,\mu}-\bar \lambda_{\xi}(\mu) u_{\xi,\mu}\\
& \quad -\Big[\mathcal{K}_{\xi,\mu}u_{\xi,\mu}-u_{\xi,\mu}+(a_2(t,x)-2c_2(t,x)v^*(t,x))u_{\xi,\mu}\\
&\qquad \quad -\bar \lambda_{\xi}(\mu) u_{\xi,\mu}+b_2(t,x)v^*(t,x)u_{\xi,\mu}(t,x)\Big]\\
&\ge \Big[a_1(t,x)-c_1(t,x)v^*(t,x)-a_2(t,x)+2c_2(t,x)v^*(t,x)-b_2(t,x) v^*(t,x) \Big]u_{\xi,\mu}\\
&\ge 0\qquad\qquad\qquad \text{(by (HL1)).}
\end{align*}
Hence $v=u_{\xi,\mu}(t,x)$ is a super-solution of
\eqref{linear-ev-eq4}.  Therefore, we must have
$$
v_{\xi,\mu}(t,x)\le u_{\xi,\mu}(t,x).
$$
By (HL1) again, $b_1(t,x)\ge c_1(t,x)$ and $b_2(t,x)\ge c_2(t,x)$. It then follows that
$$c_1(t,x) v_{\xi,\mu}(t,x)\le b_1(t,x) u_{\xi,\mu}(t,x)$$ and
 $$c_2(t,x)v_{\xi,\mu}(t,x)\le b_2(t,x) u_{\xi,\mu}(t,x).$$

Next, we assume (HL2). We claim that both $v=\frac{b_{1L}}{c_{1M}} u_{\xi,\mu}$ and
$v=\frac{b_{2L}}{c_{2M}}u_{\xi,\mu}$ are super-solutions of \eqref{linear-ev-eq4}.
In fact,
\begin{align*}
&\p_t \frac{b_{1L}}{c_{1M}}u_{\xi,\mu}(t,x)-\Big[\mathcal{K}_{\xi,\mu}\frac{b_{1L}}{c_{1M}}u_{\xi,\mu}-\frac{b_{1L}}{c_{1M}}u_{\xi,\mu}
+(a_2(t,x)-2c_2(t,x)v^*(t,x))\frac{b_{1L}}{c_{1M}}u_{\xi,\mu}\\
&\quad\qquad\quad\quad\quad -\bar \lambda_{\xi}(\mu) \frac{b_{1L}}{c_{1M}}u_{\xi,\mu}+b_2(t,x)v^*(t,x)u_{\xi,\mu}(t,x)\Big]\\
&\ge \Big[a_1(t,x)-c_1(t,x)v^*(t,x)-a_2(t,x)+2c_2(t,x)v^*(t,x)-b_2(t,x) v^*(t,x) \frac{c_{1M}}{b_{1L}}\Big]\frac{b_{1L}}{c_{1M}}u_{\xi,\mu}\\
&\ge 0\qquad\qquad\qquad \text{(by (HL2)).}
\end{align*}
Hence $v=\frac{b_{1L}}{c_{1M}} u_{\xi,\mu}$ is a super-solution
of \eqref{linear-ev-eq4}. Similarly, we can prove that
$v=\frac{b_{2L}}{c_{2M}}u_{\xi,\mu}$ is a super-solution of \eqref{linear-ev-eq4}.
Therefore, we also have
$$
v_{\xi,\mu}(t,x)\le \frac{b_{1L}}{c_{1M}} u_{\xi,\mu}(t,x)
$$
and
$$
v_{\xi,\mu}(t,x)\le \frac{b_{2L}}{c_{2M}}u_{\xi,\mu}(t,x).
$$
It then follows that
$$c_1(t,x) v_{\xi,\mu}(t,x)\le b_1(t,x) u_{\xi,\mu}(t,x)$$ and
 $$c_2(t,x)v_{\xi,\mu}(t,x)\le b_2(t,x) u_{\xi,\mu}(t,x).$$
\end{proof}

We now prove Theorem \ref{linear-determinacy-thm}.

\begin{proof}[Proof of Theorem \ref{linear-determinacy-thm}]
First of all, for any $\epsilon>0$,
let $\bar a_1(t,x)$, $\bar c_1(t,x)$,  $\bar v^*(t,x)$, and $\mu^*(\xi)$ be as in Lemma \ref{linear-determinacy-lm1}.
For any $M>0$,
let
$$
\begin{cases}
\tilde u(t,x;M)= Me^{-\mu^*(\xi) \big(x\cdot\xi- \frac{\bar \lambda_{\xi}(\mu^*(\xi))}{\mu^*(\xi)} t\big)}u_{\xi,\mu^*(\xi)}(t,x)\cr
\tilde  v(t,x;M)=Me^{-\mu^*(\xi) \big(x\cdot\xi-\frac{\bar \lambda_{\xi}(\mu^*(\xi))}{\mu^*(\xi)}t\big)} v_{\xi,\mu^*(\xi)}(t,x)
\end{cases}
$$
and
$$
\begin{cases}
u^+(t,x;M)=\min\{ u^*(t,x), \tilde u(t,x;M)\}\cr
v^+(t,x;M)=\min\{ v^*(t,x),\tilde v(t,x;M)\}.
\end{cases}
$$

We claim that, for any given $(u_0,v_0)\in X_1^+\times X_2^+$, if
$$
u_0(x)\le u^+(0,x;M),\quad v_0(x)\le v^+(0,x;M)\quad \forall\,\, x\in\RR^N,
$$
then
\begin{equation}
\label{aux-linear-eq1}
 u(t,x;u_0,v_0)\le u^+(t,x;M),\quad  v(t,x;u_0,v_0)\le v^+(t,x;M)\quad \forall\,\, 0\le t\le T,\quad x\in\RR^N.
\end{equation}

Assume first that the claim holds. For any $(u_0,v_0)\in X^+_1(\xi)\times X^+_2(\xi)$, let $M_0>0$  be such that
$$
u_0(x)\le u^+(0,x;M_0),\quad v_0(x)\le v^+(0,x;M_0)\quad \forall\,\, x\in\RR^N.
$$
Then by \eqref{aux-linear-eq1},
$$
 u(t,x;u_0,v_0)\le u^+(t,x;M_0),\quad  v(t,x;u_0,v_0)\le v^+(t,x;M_0)\quad \forall\,\, 0\le t\le T,\quad x\in\RR^N.
 $$
Let
$$
M_1=M_0e^{\bar\lambda_{\xi}(\mu^*(\xi))T}.
$$
Then
$$
u(T,x;u_0,v_0)\le u^+(0,x;M_1),\quad v(T,x;u_0,v_0)\le v^+(0,x;M_1)\quad \forall\,\, x\in\RR^N.
$$
By \eqref{aux-linear-eq1} again,
$$
u(t+T,x;u_0,v_0)=u(t,x;u(T,\cdot;u_0,v_0),v(T,\cdot;u_0,v_0))\le u^+(t,x;M_1)=u^+(t+T,x;M_0)
$$
and
$$
v(t+T,x;u_0,v_0)=v(t,x;u(T,\cdot;u_0,v_0),v(T,\cdot;u_0,v_0))\le v^+(t,x;M_1)=v^+(t+T,x;M_0)
$$
for $t\in [0,T]$ and $x\in\RR^N$.
Hence
$$
 u(t,x;u_0,v_0)\le u^+(t,x;M_0),\quad  v(t,x;u_0,v_0)\le v^+(t,x;M_0)\quad \forall\,\, 0\le t\le 2T,\quad x\in\RR^N.
 $$
Continuing the above process, we have
$$
 u(t,x;u_0,v_0)\le u^+(t,x;M_0),\quad  v(t,x;u_0,v_0)\le v^+(t,x;M_0)\quad \forall\,\, t\ge 0,\quad x\in\RR^N.
 $$
Then for any $c>\frac{\bar \lambda_{\xi}(\mu^*(\xi))}{\mu^*(\xi)}$,
$$
\limsup_{x\cdot\xi\ge ct, t\to\infty}[u(t,x;u_0,v_0)+v(t,x;u_0,v_0)]=0
$$
and hence
$$
c_{\sup}^*(\xi)\le \frac{\bar \lambda_{\xi}(\mu^*(\xi))}{\mu^*(\xi)}\le \bar c_{\inf}^*(\xi)+\epsilon.
$$
By Theorem \ref{spreading-speed-thm1}(2),
$$
c_{\inf}^*(\xi)\ge \bar c_{\inf}^*(\xi).
$$
Letting $\epsilon\to 0$, we have
$$
c_{\sup}^*(\xi)=
c_{\inf}^*(\xi)= \bar c_{\inf}^*(\xi).
$$
Thus the theorem follows.

We now prove that the claim holds. To this end,
choose $\bar M>0$ such that
\begin{equation}
\label{aux-linear-eq2-1}
\begin{cases}
\bar M-1+ a_1(t,x)-2 \tilde d\cdot b_1(t,x)  u^*(t,x)-c_1(t,x)v^*(t,x)>0\cr
\bar M-1+a_2(t,x)- \tilde d\cdot b_2(t,x) u^*(t,x)-2c_2(t,x)v^*(t,x)>0
\end{cases}
\end{equation}
for all $t\in\RR$ and $x\in\RR^N$, where
\begin{align*}
\tilde d=&\max\Big\{1, \frac{u^*(t,x)}{u^*(\tau,x)}\cdot\frac{u_{\xi,\mu^*(\xi)}(\tau,x)}{u_{\xi,\mu^*(\xi)}(t,x)}\cdot e^{\bar \lambda_{\xi}(\mu^*(\xi))(\tau-t)}, \frac{ v^*(t,x)}{ u^*(\tau,x)}\cdot\frac{u_{\xi,\mu^*(\xi)}(\tau,x)}
{v_{\xi,\mu^*(\xi)}(t,x)}\cdot e^{\bar \lambda_{\xi}(\mu^*(\xi)) (\tau-t)}\\
\\
&\qquad\qquad\qquad\qquad  \,\Big |\,  t-T\le \tau\le t,\,\, t,\tau\in\RR,\,\, x\in\RR^N\Big\}.
\end{align*}
Let
$$
\bar u(t,x;M,\bar M)= e^{\bar Mt}\tilde u(t,x;M),\quad
\bar v(t,x;M,\bar M)=e^{\bar Mt}\tilde u(t,x;M)
$$
and
$$\bar u^*(t,x;\bar M)=e^{\bar Mt} u^*(t,x),\quad v^*(t,x;\bar M)=e^{\bar Mt}v^*(t,x).
$$
Let
\begin{equation*}
\begin{cases}
\bar u^+(t,x;M,\bar M)=\min\{\bar u^*(t,x;\bar M), \bar u(t,x;M,\bar M)\}\cr
\bar v^+(t,x;M,\bar M)=\min\{\bar v^*(t,x;M),\bar v(t,x;M,\bar M)\}
\end{cases}
\end{equation*}
and
\begin{equation*}
\begin{cases}
$$
\bar u(t,x;u_0,v_0,\bar M)=e^{\bar Mt}u(t,x;u_0,v_0)\cr
 \bar v(t,x;u_0,v_0,\bar M)=e^{\bar Mt} v(t,x;u_0,v_0).
\end{cases}
\end{equation*}
It suffices to prove that, for any $(u_0,v_0)\in X_1^+\times X_2^+$, if
$$
u_0(x)\le \bar u^+(0,x;M,\bar M),\quad v_0(x)\le \bar v^+(0,x;M,\bar M)\quad \forall\,\, x\in\RR^N,
$$
then
\begin{equation}
\label{aux-linear-eq3}
\bar u(t,x;u_0,v_0,\bar M)\le \bar u^+(t,x;M,\bar M),\quad \bar v(t,x;u_0,v_0,\bar M)\le \bar v^+(t,x;M,\bar M)
\end{equation}
for $0\le t\le T$ and $x\in\RR^N$.

In the following, if no confusion occurs, we write
$\bar u(t,x;M,\bar  M)$, $\bar u^*(t,x;M)$, $\bar u^+(t,x;M,\bar M)$, $\bar v(t,x;M,\bar M)$,
$\bar v^*(t,x;M)$, and  $v^+(t,x;M,\bar M)$ as $\bar u(t,x)$, $\bar u^*(t,x)$, $\bar u^+(t,x)$, $\bar v(t,x)$,
$\bar v^*(t,x)$, and  $\bar v^+(t,x)$, respectively.

Let
$$
\bar  f(t,x,u,v)=u\Big(\bar M-1+a_1(t,x)-e^{-\bar Mt}b_1(t,x)u-c_1(t,x)v^*(t,x)+e^{-\bar Mt}c_1(t,x)v\Big)
$$
and
$$
\bar  g(t,x,u,v)=b_2(t,x)\Big(v^*(t,x)-e^{-\bar Mt}v\Big)u+v\Big(\bar M-1+a_2(t,x)-2c_2(t,x)v^*(t,x)+e^{-\bar Mt}c_2(t,x)v\Big).
$$
Note that
\begin{equation}
\label{auux-linear-eq1}
\begin{cases}
\bar f_v(t,x,u,v)\ge 0\quad \forall\,\, t\in\RR,\,\, x\in\RR^N,\,\, u\ge 0\cr
\bar g_u(t,x,u,v)\ge 0\quad \forall\,\, 0\le v\le \bar v^*(t,x),\,\, t\in\RR,\,\, x\in\RR^N.
\end{cases}
\end{equation}
By \eqref{aux-linear-eq2-1},
\begin{equation}
\label{auux-linear-eq2}
\bar f_u(t,x,u,v)\ge 0,\quad  \bar g_v(t,x,u,v)\ge 0
\end{equation}
for $0\le u\le  \tilde d \bar u^*(t,x)$, $v\ge 0$, $t\in\RR$, $x\in\RR^N$.

We first show that  for any $(t,x)\in\RR\times\RR^N$, we have
\begin{equation}
\label{auux-linear-eq3}
b_2(t,x)\bar u^+(t,x)\ge c_2(t,x)\bar v^+(t,x).
\end{equation}
If $(t,x)$ is such that $\bar v^+(t,x)=\bar v^*(t,x)$ and $\bar u^+(t,x)=\bar u^*(t,x)$, then by (HL0),
\begin{equation*}
b_2(t,x)\bar u^+(t,x)\ge c_2(t,x) \bar v^+(t,x).
\end{equation*}
If $(t,x)$ is such that $\bar v^+(t,x)=\bar v(t,x)\le \bar v^*(t,x)$, then by Lemma \ref{linear-determinacy-lm2},
$$
 b_2(t,x) \bar u^+(t,x)=b_2(t,x)\min\{\bar u(t,x),\bar u^*(t,x)\}\ge c_2(t,x) \bar v^+(t,x).
$$
If $(t,x)$ is such that $\bar u^+(t,x)=\bar u(t,x)\le \bar u^*(t,x)$, then by Lemma \ref{linear-determinacy-lm2}  again,
$$
b_2(t,x) \bar u^+(t,x)\ge c_2(t,x)\bar v(t,x)\ge c_2(t,x)\bar v^+(t,x).
$$
Hence for any $(t,x)\in\RR\times\RR^N$, \eqref{auux-linear-eq3} holds.

By \eqref{auux-linear-eq1} and \eqref{auux-linear-eq2},
\begin{align}
\label{aux-linear-eq4-2}
\begin{cases}
\bar u^*(t,x)&=\bar u^*(0,x)+\int_0^t \Big[\int_{\RR^N}\kappa (y-x)\bar u^*(\tau,y)dy+\bar f(\tau,x,\bar u^*(\tau,x),\bar v^*(\tau,x)\Big]d\tau\\
&\ge  \bar u^+(0,x)+\int_0^t \Big[\int_{\RR^N}\kappa (y-x)\bar u^+(\tau,y)dy+\bar f(\tau,x,\bar u^+(\tau,x),\bar v^+(\tau,x)\Big]d\tau\\
\bar v^*(t,x)&=\bar v^*(0,x)+\int_0^t \Big[\int_{\RR^N}\kappa (y-x)\bar v^*(\tau,y)dy+\bar g(\tau,x,\bar u^*(\tau,x),\bar v^*(\tau,x)\Big]d\tau\\
&\ge \bar v^+(0,x)+\int_0^t \Big[\int_{\RR^N}\kappa (y-x)\bar v^+(\tau,y)dy+\bar g(\tau,x,\bar u^+(\tau,x),\bar v^+(\tau,x)\Big]d\tau
\end{cases}
\end{align}
for all $(t,x)\in\RR\times \RR^N$.
Note  that if $(t,x)\in [0,T]\times\RR^N$ is such that $\bar u(t,x)=\bar u^+(t,x)\le \bar u^*(t,x)$, then
\begin{equation*}
\bar u(\tau ,x)\le e^{\bar \lambda_{\xi}(\mu^*(\xi)) (\tau-t)}\cdot\frac{ u^*(t,x)}{ u^*(\tau,x)}\cdot\frac{u_{\xi,\mu^*(\xi)}(\tau,x)}
{u_{\xi,\mu^*(\xi)}(t,x)}\cdot \bar u^*(\tau,x)\le \tilde d \cdot\bar u^*(\tau,x)
\end{equation*}
for $0\le\tau\le t$. Hence
\begin{align*}
\bar u(t,x)&=\bar u(0,x)+\int_0^t \Big[\int_{\RR^N}\kappa(y-x)\bar u(\tau,y)dy+(\bar M-1+\bar a_1(\tau,x)-\bar c_1(\tau,x)\bar v^*(\tau,x))\bar u\Big]d\tau\\
&\ge \bar  u(0,x)+\int_0^t \Big[\int_{\RR^N}\kappa(y-x)\bar u(\tau,y)dy
+\bar f(\tau,x,\bar u(\tau,x),\bar v(\tau,x))\Big]d\tau\qquad \text{(by Lemma \ref{linear-determinacy-lm2})}\\
&\ge \bar u^+(0,x)+\int_0^t \Big[\int_{\RR^N}\kappa(y-x)\bar u^+(\tau,y)dy+\bar f(\tau,x,\bar u^+(\tau,x),\bar v^+(\tau,x))\Big]d\tau\quad
\text{(by \eqref{auux-linear-eq1}, \eqref{auux-linear-eq2})}.
\end{align*}
If $(t,x)\in [0,T]\times\RR^N$ is such that $\bar v(t,x)=\bar v^+(t,x)$, then
$$
\bar u(\tau ,x)\le e^{\bar \lambda_{\xi}(\mu^*(\xi)) (\tau-t)}\cdot\frac{ v^*(t,x)}{ u^*(\tau,x)}\cdot\frac{u_{\xi,\mu^*(\xi)}(\tau,x)}
{v_{\xi,\mu^*(\xi)}(t,x)}\cdot \bar u^*(\tau,x)\le \tilde d \cdot\bar u^*(\tau,x)
$$
Hence
\begin{align*}
\bar v(t,x)&=\bar v(0,x)+\int_0^t \Big[\int_{\RR^N}\kappa(y-x)\bar v(\tau,y)dy+b_2(\tau,x)v^*(\tau,x)\bar u(\tau,x)\\
&\qquad\qquad\qquad\qquad +(\bar M-1+a_2(\tau,x)-2c_2(\tau,x)v^*(\tau,x))\bar v\Big]d\tau\\
&\ge \bar v^+(0,x)+\int_0^t \Big[\int_{\RR^N}\kappa(y-x)\bar v^+(\tau,y)dy+b_2(\tau,x)v^*(\tau,x)\bar u^+(\tau,x)\\
&\qquad\qquad\qquad\qquad +(\bar M-1+a_2(\tau,x)-2c_2(\tau,x)v^*(\tau,x))\bar v^+(\tau,x)\Big]d\tau\qquad \text{(by \eqref{aux-linear-eq2-1})}\\
&\ge \bar v^+(0,x) +\int_0^t \Big[\int_{\RR^N}\kappa(y-x)\bar v^+(\tau,y)dy
 + \bar g(\tau,x,\bar u^+(\tau,x),\bar v^+(\tau,x))\Big]d\tau\qquad \text{(by \eqref{auux-linear-eq3}).}\\
\end{align*}
It then follows that for any $(t,x)\in [0,T]\times \RR^N$,
\begin{equation}
\label{aux-linear-eq5}
\begin{cases}
\bar u^+(t,x)\ge  \bar u^+(0,x)+\int_0^t \Big[\int_{\RR^N}\kappa (y-x)\bar u^+(\tau,y)dy+\bar f(\tau,x,\bar u^+(\tau,x),\bar v^+(\tau,x)\Big]d\tau\cr
\bar v^+(t,x)\ge \bar v^+(0,x)+\int_0^t \Big[\int_{\RR^N}\kappa (y-x)\bar v^+(\tau,y)dy+\bar g(\tau,x,\bar u^+(\tau,x),\bar v^+(\tau,x)\Big]d\tau
\end{cases}
\end{equation}

For given $(u_0,v_0)$ with $u_0\le u^+(0,\cdot)$ and $v_0\le v^+(0,\cdot)$,
put
$$
\bar u_0(t,x)=e^{\bar Mt}u(t,x;u_0,v_0),\quad \bar v_0(t,x)=e^{\bar Mt} v(t,x;u_0,v_0).
$$
Then for all $t\ge 0$ and $x\in\RR^N$,
\begin{equation}
\label{aux-linear-eq6}
\begin{cases}
\bar u_0(t,x)=\bar u_0(x) +\int_0^t \Big[\int_{\RR^N}\kappa(y-x) \bar u_0(\tau,y)dy+
\bar f(\tau,x,\bar u_0(\tau,x),\bar v_0(\tau,x))\Big]d\tau\cr
\bar v_0(t,x)=\bar v_0(x) +\int_0^t \Big[\int_{\RR^N}\kappa(y-x)\bar v_0(\tau,y)dy
+\bar g(\tau,x,\bar u_0(\tau,x),\bar v_0(\tau,x))\Big]d\tau.
\end{cases}
\end{equation}

Let
$$
\tilde u(t,x)=\bar u^+(t,x)-\bar u_0(t,x),\quad \tilde v(t,x)=\bar v^+(t,x)-\bar v_0(t,x).
$$
By \eqref{aux-linear-eq5} and \eqref{aux-linear-eq6},
$$
\begin{cases}
\tilde u(t,x)\ge \tilde u(0,x)+\int_0^t \Big[\int_{\RR^N} \kappa(y-x)\tilde u(\tau,y)dy+\tilde a_1(\tau,x)\tilde u(\tau,x)+\tilde b_1(\tau,x)\tilde v(\tau,x)\Big]d\tau\cr
\tilde v(t,x)\ge\tilde v(0,x)+ \int_0^t \Big[\int_{\RR^N}\tilde \kappa(y-x)\tilde v(\tau,y)dy+\tilde a_2(\tau,x)\tilde u(\tau,x)+\tilde b_2(\tau,x)\tilde v(\tau,x)\Big]d\tau,
\end{cases}
$$
where
$$
\tilde a_1(t,x)=\bar f_u(t,x,\tilde u_1(t,x),\tilde v_1(t,x)),\quad \tilde b_1(t,x)=\bar f_v(t,x,\tilde u_1(t,x),\tilde v_1(t,x)),
$$
and
$$
\tilde a_2(t,x)=\bar g_u(t,x, \tilde u_2(t,x),\tilde v_2(t,x)),\quad
\tilde b_2(t,x)=\bar g_v(t,x,\tilde u_2(t,x),\tilde v_2(t,x))
$$
for some $\tilde u_1(t,x)$, $\tilde u_2(t,x)$ between $\bar u^+(t,x)$ and $\bar u_0(t,x)$,
and some $\tilde v_1(t,x)$, $\tilde v_2(t,x)$ between $\bar v^+(t,x)$ and $\bar v_0(t,x)$.
By \eqref{auux-linear-eq1} and \eqref{auux-linear-eq2},
$$
\tilde a_i(t,x)\ge 0,\quad \tilde b_i(t,x)\ge 0\quad \forall\,\, t\in [0,T],\quad x\in\RR^N,\quad i=1,2.
$$
Then
by comparison principle for cooperative systems,
$$
\tilde u(t,x)=\bar u^+(t,x)-\bar u_0(t,x)\ge 0,\quad \tilde v(t,x)=\bar v^+(t,x)-\bar v_0(t,x)\ge 0
$$
for $t\in [0,T]$ and $x\in\RR^N$, which implies \eqref{aux-linear-eq3} and then \eqref{aux-linear-eq1}.
This completes the proof of the claim and then the theorem.
\end{proof}

\section{Remarks on Spreading Speeds of Two Species Competition Systems with
Random and Discrete Dispersals in Periodic Habitats}

We remark that the methods developed in this paper can be applied to the study of
spreading speeds and linear determinacy for the following two species competition systems,
\begin{equation}
\label{random-dispersal-eq}
\begin{cases}
u_t=d_1\Delta u+u(a_1(t,x)-b_1(t,x)u-c_1(t,x)v),\quad x\in\RR^N\cr
v_t=d_2\Delta v+v(a_2(t,x)-b_2(t,x)u-c_2(t,x)v),\quad x\in\RR^N,
\end{cases}
\end{equation}
and
\begin{equation}
\label{discrete-dispersal-eq}
\begin{cases}
\dot u_j(t)=d_1\sum_{k\in K}(u_{j+k}(t)-u_j(t))+u_j(t)(a_1(t,j)-b_1(t,j)u_j(t)-c_1(t,j)v_j(t)),\quad j\in\ZZ^N\cr
\dot v_j(t)=d_2\sum_{k\in K}(v_{j+k}(t)-v_j(t))+v_j(t)(a_2(t,j)-b_2(t,j)u_j(t)-c_2(t,j)v_j(t)),\quad j\in \ZZ^N,
\end{cases}
\end{equation}
where $K$ is a bounded subset of $\ZZ^N$.

In particular, the definition of spreading speed interval for \eqref{main-eq} (see Definition \ref{spreading-speed-cooperative})
can be applied to \eqref{random-dispersal-eq} and \eqref{discrete-dispersal-eq}. Similar results to Theorems \ref{spreading-speed-thm1},
\ref{spreading-speed-thm2}, and \ref{linear-determinacy-thm} can also be obtained for \eqref{random-dispersal-eq} and \eqref{discrete-dispersal-eq}.

As mentioned in the introduction, the authors of \cite{YuZh} have been studying the spreading speeds and traveling waves of \eqref{random-dispersal-eq}.
The results in \cite{FaZh} and \cite{WeLeLi} apply to \eqref{random-dispersal-eq} and \eqref{discrete-dispersal-eq} with spatial homogeneous
and time periodic coefficients. One is also referred to \cite{GuWu} and references therein for
the study of traveling wave solutions of  \eqref{discrete-dispersal-eq} with space and time independent coefficients.

\end{document}